\documentclass[11pt,twoside,a4paper]{article}
\usepackage{color,amscd,amsmath,amssymb,amsthm,latexsym,stmaryrd}
\usepackage[latin1]{inputenc}
\usepackage[T1]{fontenc}   
\usepackage{stmaryrd}
\usepackage[english]{babel}
\newcommand{\tun}{\begin{picture}(5,0)(-2,-1)
\put(0,0){\circle*{2}}
\end{picture}}

\input{xy}
\xyoption{all}

\theoremstyle{plain}
\newtheorem{theo}{Theorem}
\newtheorem{lemma}[theo]{Lemma}
\newtheorem{cor}[theo]{Corollary}
\newtheorem{prop}[theo]{Proposition}
\newtheorem{defi}[theo]{Definition}

\theoremstyle{remark}
\newtheorem{remark}{Remark}
\newtheorem{notation}{Notations}
\newtheorem{example}{Example}

\newcommand{\N}{\mathbb{N}}

\newcommand{\K}{\mathbb{K}}

\newcommand{\h}{\mathcal{H}}

\newcommand{\QSym}{\mathbf{QSym}}

\newcommand{\rond}[1]{*++[o][F-]{#1}}
\newcommand{\fg}{\mathcal{FG}}
\newcommand{\ncfg}{\mathcal{NCFG}}
\newcommand{\sg}{\mathcal{SG}}
\newcommand{\ncsg}{\mathcal{NCSG}}
\newcommand{\qp}{\mathcal{QP}}
\newcommand{\p}{\mathcal{P}}
\newcommand{\bfM}{\mathbf{M}}
\newcommand{\bfH}{\mathbf{H}}
\newcommand{\bfA}{\mathbf{A}}

\begin{document}

\title{Realizations of Hopf algebras of graphs by alphabets}
\date{}
\author{Lo\"\i c Foissy\\ \\
{\small \it Fédération de Recherche Mathématique du Nord Pas de Calais FR 2956}\\
{\small \it Laboratoire de Mathématiques Pures et Appliquées Joseph Liouville}\\
{\small \it Université du Littoral Côte d'Opale-Centre Universitaire de la Mi-Voix}\\ 
{\small \it 50, rue Ferdinand Buisson, CS 80699,  62228 Calais Cedex, France}\\ \\
{\small \it Email: foissy@univ-littoral.fr}}

\maketitle

\begin{abstract}
We here give polynomial realizations of various Hopf algebras or bialgebras on Feynman graphs, graphs, 
posets or quasi-posets, that it to say injections of these objects into polynomial algebras generated by an
alphabet. The alphabet here considered are totally quasi-ordered. The coproducts are given by doubling the alphabets; 
a second coproduct is defined by squaring the alphabets, and we obtain
cointeracting bialgebras in the commutative case.
\end{abstract}



\tableofcontents

\section*{Introduction}

Some combinatorial Hopf algebras admit a polynomial realization, which gives an efficient way 
to prove the existence of the coproduct
and more structures, see \cite{DHNT,NT,FNT,DLNTT,Maurice}. 
Let us explicit a well-known example. The algebra $\QSym$ of quasi-symmetric functions
has a basis $(M_{a_1,\ldots,a_k})$ indexed by compositions, that is to say finite sequences of positive integers.
\begin{enumerate}
\item For any totally ordered alphabet $X$, let us consider the following elements of the ring of formal series $\K[[X]]$ generated by $X$:
\[M_{(a_1,\ldots,a_k)}(X)=\sum_{x_1<\ldots<x_k\:\mbox{\scriptsize in }X}x_1^{a_1}\ldots x_n^{a_n}.\]
This defines a map from $\QSym$ to $\K[[X]]$, injective if, and only if, $X$ is infinite. 
This map is an algebra morphism. For example, if $a,b\geq 1$, for any totally ordered alphabet $X$:
\begin{align*}
M_{(a)}(X)M_{(b)}(X)&=\sum_{x,y\in X} x^a y^b\\
&=\sum_{x<y}x^ay^b+\sum_{y<x}x^ay^b+\sum_x x^{a+b}\\
&=M_{(a,b)}(X)+M_{(b,a)}(X)+M_{(a+b)}(X),
\end{align*}
and, in $\QSym$:
\[M_{(a)}M_{(b)}=M_{(a,b)}+M_{(b,a)}+M_{(a+b)}.\]
\item If $X$ and $Y$ are totally ordered alphabets, then $X\sqcup Y$ is too, 
the elements of $X$ being smaller than the elements of $Y$.
Identifying $\K[[X\sqcup Y]]$ with a subalgebra of $\K[[X]]\otimes \K[[Y]]$, we define a coproduct on $\QSym$ by:
\[\Delta(M_{(a_1,\ldots,a_k)})(X,Y)=M_{(a_1,\ldots,a_k)}(X\sqcup Y).\]
For example:
\begin{align*}
\Delta(M_{(a,b)})(X,Y)&=M_{(a,b)}(X\sqcup Y)\\
&=\sum_{x<y \:\mbox{\scriptsize in} X}x^ay^b+\sum_{x<y\: \mbox{\scriptsize in} Y}x^ay^b+\sum_{(x,y)\in X\times Y}x^ay^b\\
&=M_{(a,b)}(X)+M_{(a,b)}(Y)+M_{(a)}(X)M_{(b)}(Y),
\end{align*}
and, in $\QSym$:
\[\Delta(M_{(a,b)})=M_{(a,b)}\otimes 1+1\otimes M_{(a,b)}+M_{(a)}\otimes M_{(b)}.\]
The coassociativity of $\Delta$ is easily obtained from the equality $(X\sqcup Y)\sqcup Z=X\sqcup (Y\sqcup Z)$.
\item If $X$ and $Y$ are totally ordered alphabets, then $XY=X\times Y$ is too, with the lexicographic order.
We consider $\K[[XY]]$ as a subalgebra of $\K[[X]]\otimes \K[[Y]]$, identifying $(x,y)$ with $x\otimes y$.
We can define a second coproduct on $\QSym$ by:
\[\delta(M_{(a_1,\ldots,a_k)})(X,Y)=M_{(a_1,\ldots,a_k)}(XY).\]
For example:
\begin{align*}
\delta(M_{(a,b)})(X,Y)&=M_{(a,b)}(XY)\\
&=\sum_{\substack{x<x' \:\mbox{\scriptsize in} X, \\ y,y'\in Y}}x^ay^ax'^by'^b
+\sum_{\substack{x\in X,\:y<y' \:\mbox{\scriptsize in} X, }}x^ay^ax^by'^b\\
&=M_{(a,b)}(X)(M_{(a,b)}(Y)+M_{(b,a)}(Y)+M_{(a+b)}(Y))+M_{(a+b)}(X)M_{(a,b)}(Y),
\end{align*}
and, in $\QSym$:
\[\delta(M_{(a,b)})=M_{(a,b)}\otimes (M_{(a,b)}+M_{(b,a)}+M_{(a+b)})+M_{(a+b)}\otimes M_{(a,b)}.\]
The coassociativity of $\delta$ is easily obtained from the equality $(XY) Z=X(YZ)$.
Moreover:
\[(X\sqcup Y)Z=(XZ)\sqcup (YZ).\]
This implies that in $\QSym$:
\[(\Delta\otimes Id)\circ \delta=m_{1,3,24}\circ(\delta \otimes \delta)\circ \Delta,\]
where $m_{1,3,24}:\QSym^{\otimes 4}\longrightarrow \QSym^{\otimes 3}$ sends $x\otimes y\otimes z\otimes t$ to $x\otimes z\otimes yt$.
This means that the Hopf algebra $(\QSym,m,\Delta)$ is a Hopf algebra 
in the category of right comodules over the bialgebra $(\QSym,m,\delta)$, the coaction being $\delta$ itself: 
we call this a pair of bialgebras in cointeraction.
\end{enumerate}
For other examples of such objects and applications, see \cite{Manchon,BelhajManchon,Foissy2,Foissy3,Foissy4}.\\

We here give other examples of cointeracting bialgebras coming from the manipulation of alphabets and polynomial realizations.
We use here totally quasi-ordered alphabets, that is to say sets with a total transitive reflexive (but not necessarily antisymmetric) relation.
The associated algebras $A_q(X)$ are slightly more complicated, see Definition \ref{defi9}. Their different sets of generators
allows to polynomially realize Feynman graphs (one set for vertices, one set for internal edges, one set for incoming half-edges
and a last one for outgoing half-edges); this gives a family of products $\cdot_q$ on the space $H_{\fg}$ generated by isoclasses of Feynman graphs,
indexed by a scalar $q$ (Theorem \ref{theoprod}). 
If $F$ and $G$ are two Feynman graphs, $F\cdot_q G$ is a sum of graphs obtained by gluing together vertices of 
$F$ and $G$; in particular, if $q=0$, this is reduced to the disjoint union of $F$ and $G$. 
The trick of doubling the alphabet gives $\h_{\fg}$ a coproduct $\Delta$, given by ideals (Theorem \ref{theocop}),
and the trick of squaring the alphabet gives it a second coproduct $\delta$; 
we obtain in this way a pair of cointeracting bialgebras (Corollary \ref{corDelta}).
For example, for the following graph:
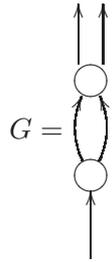
\[G=\mbox{\parbox{3mm}{\xymatrix{\\
\rond{\:}\ar@<1ex>[u] \ar@<-1ex>[u]\\
\rond{\:}\ar@/^/[u]\ar@/_/[u]\\
\ar[u]}}},\]
we obtain:
\begin{align*}
\Delta(G)&=G\otimes 1+1\otimes G+
\mbox{\parbox{3mm}{\xymatrix{\\
\rond{\:}\ar@<1ex>[u] \ar@<-1ex>[u]\\
\ar[u]}}}\otimes \mbox{\parbox{3mm}{\xymatrix{\\
\rond{\:}\ar@<1ex>[u] \ar@<-1ex>[u]\\
\ar@<1ex>[u] \ar@<-1ex>[u]}}},&
\delta(G)
&=G\otimes\mbox{\parbox{3mm}{\xymatrix{&\\
\rond{\:}\ar@<1ex>[u] \ar@<-1ex>[u]&\rond{\:}\ar@<1ex>[u] \ar@<-1ex>[u]\\
\ar[u]&\ar@<1ex>[u] \ar@<-1ex>[u]}}}
+\mbox{\parbox{3mm}{\xymatrix{\rond{\:}\\ \ar[u]} }}
\mbox{\parbox{3mm}{\xymatrix{\\ \rond{\:}\ar@<1ex>[u] \ar@<-1ex>[u]}}}
\otimes G.
\end{align*}
The coproduct $\delta$  is similar to the Connes-Kreimer's one $\Delta_{CK}$ \cite{CK1,CK2,CK3,CK4,CK5,CK6,Manchon}, but slightly different.
For example:
\[\Delta_{CK}(G)
=G\otimes\mbox{\parbox{3mm}{\xymatrix{&\\
\rond{\:}\ar@<1ex>[u] \ar@<-1ex>[u]&\rond{\:}\ar@<1ex>[u] \ar@<-1ex>[u]\\
\ar[u]&\ar@<1ex>[u] \ar@<-1ex>[u]}}}
+\mbox{\parbox{3mm}{\xymatrix{\\
\rond{\:}\ar@<1ex>[u] \ar@<-1ex>[u]\\
\ar[u]}}}\otimes G.\]
We shall then consider several quotients of $A_q(X)$, leading to quotient bialgebras of $\h_{\fg}$. 
We obtain in this way a polynomial realization of a Hopf algebra of simple oriented graphs $\h_{\sg}$, and then a polynomial realization
of the Hopf algebra on quasi-posets $\h_{\qp}$ of \cite{FM,FMP1,FMP2,Foissy2}.
Restricting to ordered alphabets, instead of quasi-ordered alphabets, we obtain  quotients bialgebras, namely
$H_\ncfg$ based on  Feynman graphs with no cycle in Theorem \ref{theofgwc} , $H_\ncsg$ on simple oriented graphs with no cycle
in Theorem \ref{theosgwc}, and $H_\p$ on posets in Theorem \ref{theoposets}, obtaining diagrams of Hopf algebras:
\[\xymatrix{(H_\fg,._q,\Delta)\ar@{->>}[r]^S\ar@{->>}[d]_T&(H_\sg,._q,\Delta)\ar@{->>}[r]^P\ar@{->>}[d]^T&(H_\qp,._q,\Delta)\ar@{->>}[d]^T\\
(H_\ncfg,._q,\Delta)\ar@{->>}[r]_S&(H_\ncsg,._q,\Delta)\ar@{->>}[r]_P&(H_\p,._q,\Delta)}\]
We also show that these Hopf algebra admit noncommutative versions, 
replacing the algebras $A_q(X)$ by a noncommutative analogue.
The last paragraph is devoted to the description of the dual Hopf algebra of posets, using the notion of system of edge between two posets.

\section{Operations on alphabets}

All the proofs of this section are elementary and left to the reader.

\subsection{Quasi-ordered alphabets}

\begin{defi}
A quasi-ordered alphabet is a pair $(X,\leq_X)$, where $X$ is an alphabet and $\leq_X$ is a total quasi-order on $X$, that is to say a relation on $X$ such that:
\begin{align*}
&\forall i\in X,&&i\leq_X i&\mbox{(reflexivity)},\\
&\forall i,j,k\in X,&&(i\leq_X j) \mbox{ and }(j\leq_X k) \Longrightarrow i\leq_X k&\mbox{(transitivity)},\\
&\forall i,j\in X,&&i\leq_X j\mbox{ or }j\leq_X i&\mbox{(totality)}.
\end{align*}
If $\leq_X$ is an order, we shall say that $(X,\leq_X)$ is an ordered alphabet.
\end{defi}

\begin{notation}
Let $(X,\leq_X)$ be a quasi-ordered alphabet. We define an equivalence $\sim_X$ on X by:
\begin{align*}
&\forall i,j\in X,&i\sim_X j\mbox{ if } (i\leq_X j) \mbox{ and }(j\leq_X i).
\end{align*}
For all $i,j\in X$, we shall denote $i<_X j$ if $i\leq_X j$ and not $j\leq_X i$.
\end{notation}

\subsection{Disjoint union}

\begin{prop}
Let $(X,\leq_X)$ and $(Y,\leq_Y)$ be two quasi-ordered alphabets. The set $X\sqcup Y$ is given a relation $\leq_{X\sqcup Y}$:
\begin{align*}
&\forall i,j\in X\sqcup Y,& i\leq_{X\sqcup Y} j\mbox{ if }&(i,j\in X \mbox{ and }i\leq_X j)\\
&&& \mbox{ or }(i,j\in Y \mbox{ and }i\leq_y j)\\
&&&\mbox{ or }(i\in X,\: j\in Y). 
\end{align*}
Then $(X\sqcup Y,\leq_{X\sqcup Y})$ is a quasi-ordered alphabet.
\end{prop}

\begin{remark} 
If $X$ and $Y$ are ordered alphabets, then $X\sqcup Y$ is also ordered.
\end{remark}

\begin{lemma}\begin{enumerate}
\item Let $X,Y$ be quasi-ordered alphabets. 
\begin{align*}
&\forall i,j \in X\sqcup Y,&
i\sim_{X\sqcup Y} j\Longleftrightarrow &(i,j\in X\mbox{ and }i\sim_X j)\mbox{ or }(i,j\in Y\mbox{ and } i\sim_Y j).
\end{align*} 
\item Let $X$, $Y$ and $Z$ be quasi-ordered alphabets. Then:
\[(X\sqcup Y)\sqcup Z=X\sqcup(Y\sqcup Z).\]
\end{enumerate} \end{lemma}

\subsection{Product}

\begin{prop}
Let $(X,\leq_X)$ and $(Y,\leq_Y)$ be two quasi-ordered alphabets. The set $XY=X\times Y$ is given a relation $\leq_{XY}$ in the following way:
\begin{align*}
&\forall i,i'\in X,\: j,j'\in Y,& (i,j)\leq_{XY} (i',j')\mbox{ if }&(i\sim_X i' \mbox{ and }j\leq_Y j')\mbox{ or }(i<_X i').
\end{align*}
Then $(XY,\leq_{XY})$ is a quasi-ordered alphabet, which we denote by $XY$.
\end{prop}

\begin{remark} 
If $X$ and $Y$ are ordered alphabets, then $XY$ is also ordered, and $\leq_{XY}$ is the lexicographic order.
\end{remark}

\begin{lemma}
\begin{enumerate}
\item Let $X$ and $Y$ be quasi-ordered alphabets.
\begin{align*}
&\forall i,i'\in X,\:j,j'\in Y,&
(i,j)\sim_{XY} (i',j')&\Longleftrightarrow (i\sim_X i')\mbox{ and }(j\sim_Y j').
\end{align*}
\item Let $X$, $Y$ and $Z$ be quasi-ordered alphabets. Then:
\begin{align*}
(XY)Z&=X(YZ),&(X\sqcup Y)Z&=(XZ)\sqcup (YZ),&X(Y\sqcup Z)&=(XY)\sqcup (XZ).
\end{align*}\end{enumerate}\end{lemma}

\section{Algebras attached to alphabets}

\subsection{Definition}

\begin{defi}\label{defi9}
Let $X$ be a quasi-ordered alphabet and let $q\in \K$. We put:
\begin{align*}
A_q(X)&=\frac{\K[x_i, i\in X]\:[[x_{i,j}, i,j \in X,i\leq_X j]]\:[[x_{-\infty,j}, j\in X]]\:[[x_{i,+\infty}, i\in X]]}
{\langle x_i^2=qx_i,  i\in X\rangle},&\\
\bfA_q(X)&=\frac{\K\langle x_i, i\in X\rangle\:[[x_{i,j}, i,j \in X,i\leq_X j]]\:[[x_{-\infty,j}, j\in X]]\:
[[x_{i,+\infty}, i\in X]]}
{\langle x_iPx_i=qx_iP,  i\in X,\: P\in \bfA_q(X)\rangle}.
\end{align*}
Both of them are given their usual topology of rings of formal series.
\end{defi}

Elements of $A_q(X)$ are formal infinite spans of monomials
\[M=\prod_{i\in X} x_i^{\epsilon_i} \prod_{i\leq_X j} x_{i,j}^{\alpha_{i,j}}\prod_{j\in X} x_{-\infty,j}^{\beta_j}
\prod_{i\in X} x_{i,+\infty}^{\gamma_i},\]
where $\epsilon_i\in\{0,1\}$, $\alpha_{i,j},\beta_j,\gamma_i \in \N$, with only a finite number of them non-zero.
Elements of $\bfA_q(X)$ are formal infinite spans of monomials
\[M=x_{i_1}\ldots x_{i_k} \prod_{i\leq_X j} x_{i,j}^{\alpha_{i,j}}\prod_{j\in X} x_{-\infty,j}^{\beta_j}
\prod_{i\in X} x_{i,+\infty}^{\gamma_i},\]
where $i_1,\ldots,i_k$ are elements of $X$, all distinct, $\alpha_{i,j},\beta_j,\gamma_i \in \N$, with only a finite number of them non-zero.

\subsection{Doubling the alphabets}

\begin{prop}
Let $X,Y$ be two quasi-ordered alphabets. 
We define a continuous algebra morphism $\Delta_{X,Y}$ from $A_q(X\sqcup Y)$ to $A_q(X)\otimes A_q(Y)$ 
or from $\bfA_q(X\sqcup Y)$ to $\bfA_q(X)\otimes \bfA_q(Y)$ by:
\begin{align}
\label{E1} \Delta_{X,Y}(x_i)&=\begin{cases}
x_i\otimes 1\mbox{ if }i\in X,\\
1\otimes x_i\mbox{ if }i\in Y;
\end{cases}&
\Delta_{X,Y}(x_{i,j})&=\begin{cases}
x_{i,j}\otimes 1\mbox{ if }i,j\in X,\\
1\otimes x_{i,j}\mbox{ if }i,j\in Y,\\
x_{i,\infty}\otimes x_{-\infty,j}\mbox{ if }i\in X,\: j\in Y;
\end{cases}\\
\nonumber
\Delta_{X,Y}(x_{i,\infty})&=\begin{cases}
x_{i,\infty}\otimes 1\mbox{ if }i\in X,\\
1\otimes x_{i,\infty}\mbox{ if }i\in Y;
\end{cases}&
\Delta_{X,Y}(x_{-\infty,j})&=\begin{cases}
x_{-\infty,j}\otimes 1\mbox{ if }j\in X,\\
1\otimes x_{-\infty,j}\mbox{ if }j\in Y.
\end{cases}
\end{align}
\end{prop}

\begin{proof} In the commutative case, we have to check that for any 
$i\in X\sqcup Y$, $\Delta_{X,Y}(x_i^2-qx_i)=0$. Indeed:
\[\Delta_{X,Y}(x_i^2-qx_i)=\begin{cases}
(x_i^2-qx_i)\otimes 1=0\mbox{ if }i\in X,\\
1\otimes (x_i^2-qx_i)=0\mbox{ if }i\in Y.
\end{cases}\]
So $\Delta_{X,Y}$ is well-defined. The proof in the noncommutative case is similar. \end{proof}

\begin{prop}\label{double}
Let $X$, $Y$ and $Z$ be quasi-ordered alphabets. Then:
\[(\Delta_{X,Y}\otimes Id)\circ \Delta_{X\sqcup Y,Z}=(Id \otimes \Delta_{Y,Z})\circ \Delta_{X,Y\sqcup Z},\]
seen as morphisms from $A_q(X\sqcup Y\sqcup Z)$ to $A_q(X)\otimes A_q(Y)\otimes A_q(Z)$, or
from $\bfA_q(X\sqcup Y\sqcup Z)$ to $\bfA_q(X)\otimes \bfA_q(Y)\otimes \bfA_q(Z)$.
\end{prop}

\begin{proof} It is enough to apply these two algebra morphisms on generators. We find:
\[(\Delta_{X,Y}\otimes Id)\circ \Delta_{X\sqcup Y,Z}(x_i)=(Id \otimes \Delta_{Y,Z})\circ \Delta_{X,Y\sqcup Z}(x_i)
=\begin{cases}
x_i\otimes 1\otimes 1\mbox{ if }i\in X,\\
1\otimes x_i\otimes 1\mbox{ if }i\in Y,\\
1\otimes 1\otimes x_i\mbox{ if }i\in Z.
\end{cases}\]
When applied to $x_{i,j}$, we find for both of them:
\[\begin{array}{c|c|c|c|c}
i\setminus j&\in X&\in Y&\in Z&\infty\\
\hline -\infty&x_{-\infty,j}\otimes 1\otimes 1&1\otimes x_{-\infty,j}\otimes 1&1\otimes 1\otimes x_{-\infty,j}&\times\\
\hline \in X&x_{i,j}\otimes 1\otimes 1&x_{i,\infty}\otimes x_{-\infty,j}\otimes 1&x_{i,\infty}\otimes 1\otimes x_{-\infty,j}&x_{i,\infty}\otimes 1\otimes 1\\
\hline \in Y&\times&1\otimes x_{i,j}\otimes 1&1\otimes x_{i,\infty}\otimes x_{-\infty,j}&1\otimes x_{i,\infty}\otimes 1\\
\hline \in Z&\times&\times&1\otimes 1\otimes x_{i,j}&1\otimes 1\otimes x_{i,\infty}
\end{array}\]	
So these morphisms are equal. \end{proof}

\subsection{Squaring the alphabets}

\begin{prop}
Let $X,Y$ be two nonempty quasi-ordered alphabets. There exists a unique continuous algebra morphism $\delta_{X,Y}$ from $A_q(XY)$ 
to $A_{q_1}(X)\otimes A_{q_2}(Y)$, or from $\bfA_q(XY)$ to $\bfA_{q_1}(X)\otimes \bfA_{q_2}(Y)$
such that:
\begin{align}
\label{E2} \delta_{X,Y}(x_{(i,i')})&=x_i\otimes x_{i'},&
\delta_{X,Y}(x_{-\infty,(j,j')})&=x_{-\infty,j}\otimes x_{-\infty,j'},\\
\nonumber \delta_{X,Y}(x_{(i,i'),\infty})&=x_{i,\infty}\otimes x_{i',\infty},&
\delta_{X,Y}(x_{(i,i'),(j,j')})&=\begin{cases}
x_{i,j}\otimes x_{i',\infty}x_{-\infty,j'}\mbox{ if }i<_X j,\\
1\otimes x_{i',j'}\mbox{ if }i\sim_{X} j\mbox{ and }i'\leq_{Y} j'.
\end{cases}
\end{align}
if, and only if, $q=q_1q_2$.
\end{prop}

\begin{proof} In the commutative case, we have to check that for any $(i,i')\in XY$, 
$\delta_{X,Y}(x_{(i,i')}^2-qx_{(i,i')})=0$. We compute:
\begin{align*}
\delta_{X,Y}(x_{(i,i')}^2-qx_{(i,i')})&=(x_i\otimes x_{i'})^2-q x_i\otimes x_{i'}\\
&=x_i^2\otimes x_{i'}^2-q x_i\otimes x_{i'}\\
&=q_1q_2 x_i\otimes x_{i'}-qx_i\otimes x_{i'}\\
&=(q_1q_2-q)x_i\otimes x_{i'}.
\end{align*}
So this holds if, and only if, $q_1q_2=q$. 
The noncommutative case is proved in the same way. \end{proof}

\begin{remark}
 In particular, if $q=q_1=q_2$, $\delta_{X,Y}$ exists if, and only if, $q=1$ or $q=0$. 
\end{remark}

\begin{lemma}
Let $q_1,q_2,q_3 \in \K$ and let $X$, $Y$ and $Z$ be quasi-ordered alphabets. The following diagrams commute:
\begin{align*}
&\xymatrix{A_{q_1q_2q_3}(XYZ)\ar[r]^{\delta_{X,YZ}} \ar[d]_{\delta_{XY,Z}}& 
A_{q_1 q_2}(XY)\otimes A_{q_3}(Z) \ar[d]^{\delta_{X,Y}\otimes Id}\\
A_{q_1}(X)\otimes A_{q_2 q_3}(YZ)\ar[r]_{Id\otimes \delta_{Y,Z}}&A_{q_1}(X)\otimes A_{q_2}(Y)\otimes A_{q_3}(Z)}\\ \\
&\xymatrix{\bfA_{q_1q_2q_3}(XYZ)\ar[r]^{\delta_{X,YZ}} \ar[d]_{\delta_{XY,Z}}& 
\bfA_{q_1 q_2}(XY)\otimes \bfA_{q_3}(Z) \ar[d]^{\delta_{X,Y}\otimes Id}\\
\bfA_{q_1}(X)\otimes \bfA_{q_2 q_3}(YZ)\ar[r]_{Id\otimes \delta_{Y,Z}}&\bfA_{q_1}(X)
\otimes \bfA_{q_2}(Y)\otimes \bfA_{q_3}(Z)}
\end{align*}
\end{lemma}

\begin{proof} As these two maps are algebra morphisms, it is enough to prove that they coincide 
on generators of $A_{q_1q_2q_3}(XYZ)$ or $\bfA_{q_1q_2q_3}(XYZ)$. 
Let $i,i'\in X$, $j,j'\in Y$, $k,k'\in Z$.
\begin{itemize}
\item Both send $x_{(i,j,k)}$ to $x_i\otimes x_j\otimes x_k$.
\item Both send $x_{(i,j,k),\infty}$ to $x_{i,\infty}\otimes x_{j,\infty} \otimes x_{k,\infty}$.
\item Both send $x_{-\infty,(i',j',k')}$ to $x_{-\infty,i'}\otimes x_{-\infty,j'}\otimes x_{-\infty,k'}$.
\item Both send $x_{(i,j,k),(i',j'k')}$ to $\begin{cases}
x_{i,i'}\otimes x_{j,\infty}x_{-\infty,j'}\otimes x_{k,\infty}x_{-\infty,k'}\mbox{ if }i<_X i',\\
1\otimes x_{j,j'}\otimes x_{k,\infty}x_{-\infty,k'}\mbox{ if }i\sim_X i' \mbox{ and } j<_Y j',\\
1\otimes 1\otimes x_{k,k'} \mbox{ if }i\sim_X i'\mbox{ and }j\sim_Y j'.
\end{cases}$
\end{itemize}
So $(\delta_{X,Y}\otimes Id)\circ \delta_{XY,Z}=(Id \otimes \delta_{Y,Z})\circ \delta_{X,YZ}$. \end{proof}

\begin{lemma}\label{propcoaction}
Let $q_1,q_2\in \K$ and let $X$, $Y$, $Z$ be quasi-ordered alphabets. The following diagram commutes:
\[\xymatrix{A_{q_1q_2}((X\sqcup Y)Z)=A_{q_1q_2}(XZ\sqcup YZ)\ar[r]^{\hspace{5mm}\Delta_{XZ,YZ}}
\ar[dd]_{\delta_{X\sqcup Y,Z}}&A_{q_1q_2}(XZ)\otimes A_{q_1q_2}(YZ)\ar[d]^{\delta_{XZ}\otimes \delta_{YZ}}\\
&A_{q_1}(X)\otimes A_{q_2}(Z)\otimes A_{q_1}(Y)\otimes A_{q_2}(Z)\ar[d]^{m_{1,3,24}}\\
A_{q_1}(X\sqcup Y)\otimes A_{q_2}(Z)\ar[r]_{\Delta_{X,Y}\otimes Id}&A_{q_1}(X)\otimes A_{q_1}(Y)\otimes A_{q_2}(Z)}\]
where:
\[m_{1,3,24}:\left\{\begin{array}{rcl}
A_{q_1}(X)\otimes A_{q_2}(Z)\otimes A_{q_1}(Y)\otimes A_{q_2}(Z)&\longrightarrow&A_{q_1}(X)\otimes A_{q_1}(Y)\otimes A_{q_2}(Z)\\
x\otimes z_1\otimes y\otimes z_2&\longrightarrow&x\otimes y\otimes z_1z_2.
\end{array}\right.\]
\end{lemma}

\begin{proof} \textit{First step}. Let us prove that $m_{1,3,24}$ is an algebra morphism. Let $X=x\otimes z_1\otimes y\otimes z_2$ and
$X'=x'\otimes z_1'\otimes y'\otimes z_2'$ in $A_{q_1}(X)\otimes A_{q_2}(Z)\otimes A_{q_1}(Y)\otimes A_{q_2}(Z)$.
\begin{align*}
m_{1,3,24}(XX')&=m_{1,3,24}(xx'\otimes z_1z_1'\otimes yy'\otimes z_2z_2')\\
&=xx'\otimes yy'\otimes z_1z_1'z_2z_2';\\
m_{1,3,24}(X)m_{1,3,24}(X')&=(x\otimes y\otimes z_1z_2)(x'\otimes y'\otimes z_1'z_2')\\
&=xx'\otimes yy'\otimes z_1z_2z'_1z'_2.
\end{align*}
As $A_{q_1q_2}(Z)$ is commutative, $m_{1,3,24}(XX')=m_{1,3,24}(X)m_{1,3,24}(X')$.\\

\textit{Second step.} By composition, both
$(\delta_{X,Y}\otimes Id)\circ \delta_{X\sqcup Y,Z}$ and $(Id \otimes \delta_{Y,Z})\circ \delta_{X,Y\sqcup Z}$ are algebra morphisms:
it is enough to prove that they coincide on the generators of $A_{q_1q_2}((X\sqcup Y)Z)$ or $\bfA_{q_1q_2}((X\sqcup Y)Z)$. 
Let $i,i'\in X$, $j,j'\in Y$ and $k,k'\in Z$.
\begin{itemize}
\item Both send $x_{(i,k)}$ to $x_i\otimes 1\otimes x_k$ and $x_{j,k}$ to $1\otimes x_j\otimes x_k$.
\item Both send $x_{(i,k),\infty}$ to $x_{i,\infty}\otimes 1\otimes x_{k,\infty}$ and $x_{(j,k),\infty}$ to $1\otimes x_{j,\infty}\otimes x_{k,\infty}$.
\item Both send $x_{-\infty,(i',k')}$ to $x_{-\infty,i'}\otimes 1\otimes x_{-\infty,k'}$ and $x_{-\infty,(j',k')}$ to $1\otimes x_{-\infty,j'}\otimes x_{-\infty,k'}$.
\item Both send $x_{(i,k),(i',k')}$ to $\begin{cases}
x_{i,i'}\otimes 1\otimes x_{k,\infty}x_{-\infty,k'} \mbox{ if }i<_X i',\\
1\otimes 1\otimes x_{k,k'} \mbox{ if }i\sim_X i'.
\end{cases}$
\item Both send  $x_{(j,k),(j',k')}$ to $\begin{cases}
1\otimes x_{j,j'}\otimes x_{k,\infty}x_{-\infty,k'} \mbox{ if }j<_Y j',\\
1\otimes 1\otimes x_{k,k'} \mbox{ if }j\sim_Y j'.
\end{cases}$
\item Both send $x_{(i,k),(j',k')}$ to $x_{i,\infty}\otimes x_{-\infty,j}\otimes x_{k,\infty}x_{-\infty,k'}$.
\end{itemize}
Therefore, they are equal. \end{proof}

\begin{remark}
This does not work for morphisms 
\begin{align*}
&\bfA_{q_1q_2}((X\sqcup Y)Z)=\bfA_{q_1q_2}((XZ)\sqcup (YZ))\longrightarrow \bfA_{q_1}(X)\otimes \bfA_{q_1}(Y)\otimes \bfA_{q_2}(Z).
\end{align*}
But, if we put $\rho_{X,Y}=(Id \otimes p_Y)\circ \delta_{X,Y}:\bfA_{q_1q_2}(X\sqcup Y)\longrightarrow \bfA_{q_1}(X)\otimes A_{q_2}(Y)$,
where $p_Y$ is the canonical surjection from $\bfA_ {q_2}(Y)$ to $A_{q_2}(Y)$, then
\[m_{1,3,24}\circ (\rho_{X,Z}\otimes \rho_{Y,Z})\circ \Delta_{XZ,YZ}
=(\Delta_{X,Y}\otimes Id)\circ\rho_{X\sqcup Y,Z},\]
seen as morphisms 
\begin{align*}
&\bfA_{q_1q_2}((X\sqcup Y)Z)=\bfA_{q_1q_2}((XZ)\sqcup (YZ))\longrightarrow \bfA_{q_1}(X)\otimes \bfA_{q_1}(Y)\otimes A_{q_2}(Z).
\end{align*}
The proof is identical to the one of Lemma \ref{propcoaction}.
\end{remark}

\section{Feynman graphs}

\subsection{Definition}

\begin{defi}
A Feynman graph $G$ is given by:
\begin{itemize}
\item A non-empty, finite set $HE(G)$ of half-edges, with a map $type_G:HE(G)\longrightarrow \{out,in\}$.
\item A non-empty, finite set $V(G)$ of vertices.
\item An incidence map for half-edges, that is to say an involution $i_G:HE(G)\longrightarrow HE(G)$.
\item A source map for half-edges, that is to say a map $s_G:HE(G)\longrightarrow V(G)$.
\end{itemize}
The incidence rule must be respected:
\begin{align*}
&\forall e\in HE(G),&\:(i_G(e)\neq e)\Longrightarrow type_G(e)\neq type_G\circ i_G(e).
\end{align*}
The set of external half-edges of $G$ is:
\[Ext(G)=\{e \mid e\in HE, i_G(e)=e\}.\] 
The set of Feynman graphs is denoted by $\fg$.
\end{defi}

We shall use graphical representations of Feynman graphs: vertices are represented by $\xymatrix{\rond{}}$,
half-edges of type $out$ are represented by 
$\xymatrix{\rond{}\ar[r]&}$, half-edges of type $in$ by $\xymatrix{\ar[r]&\rond{}}$; the incidence map glues two such half-edges
to obtain an oriented internal edge of the Feynman graph $\xymatrix{\rond{}\ar[r]&\rond{}}$. 
In the sequel, we shall write \emph{Feynman graph} instead of \emph{isoclass of Feynman graphs}. 
We shall also consider \emph{ordered Feynman graphs}, that is to say Feynman graphs such that the set of vertices is given a total order.

\begin{remark} 
We restraint ourselves to Feynman graphs with a unique type of edges. It is possible to do the same
for several type of edges, adding generators to the algebras associated to totally quasi-ordered alphabets.
\end{remark}

\begin{example}\label{ex1}
Here are examples of Feynman graphs:
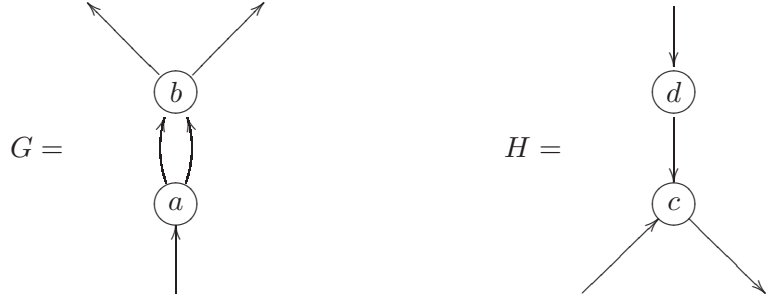
\begin{align*}
G&=\mbox{\parbox{3mm}{\xymatrix{&&\\
&\rond{b}\ar[lu] \ar[ru]&\\
&\rond{a}\ar@/^/[u]\ar@/_/[u]&\\
&\ar[u]&}}}&
H&=\mbox{\parbox{3mm}{\xymatrix{&\ar[d]&\\
&\rond{d}\ar[d]&\\
&\rond{c}\ar[dr]&\\
\ar[ru]&&}}}
\end{align*}\end{example}

\begin{defi}
Let $G_1,\ldots,G_k$ be Feynman graphs and let $\sigma:V(G_1)\sqcup\ldots \sqcup V(G_k) \twoheadrightarrow C$ be a surjective map.
We define a Feynman graph $G=\sigma(G_1,\ldots,G_k)$ in the following way:
\begin{itemize}
\item $V(\sigma(G_1,\ldots,G_k))=C$.
\item The set of half-edges of $\sigma(G_1,\ldots,G_k)$ is:
\begin{align*}
HE(G)&=\bigsqcup_{i=1}^k HE(G_i)\setminus\{e\in HE(G_i)\mid i_{G_i}(e)\neq e,\: \sigma \circ s_{G_i}(e)
=\sigma\circ s_{G_i}\circ i_{G_i}(e)\}.
\end{align*}
\item If $e\in HE(G_i)\cap HE(\sigma(G_1,\ldots,G_k))$, then:
\begin{align*}
type_G(e)&=type_{G_i}(e),&i_G(e)&=i_{G_i}(e),&s_G(e)&=\sigma\circ s_{G_i}(e).
\end{align*}\end{itemize}\end{defi}

Roughly speaking, $G\sqcup_\sigma H$ is obtained by identifying the vertices of the disjoint union of Feynman graphs
$G_1\ldots G_k$ with the same image by $\sigma$, and deleting the loops created in this process.
In particular, if $\sigma=Id_{V(G_1)\sqcup \ldots \sqcup V(G_k)}$, or more generally if $\sigma$ is bijective, then $\sigma(G_1,\ldots,G_k)$ is (isomorphic to)
the disjoint union $G_1\ldots G_k$. \\

\begin{remark}
If $G_1,\ldots,G_k$ are ordered Feynman graphs, then $G_1\ldots G_k$ is also ordered: if $x,y\in V(G_1)\sqcup \ldots \sqcup V(G_k)$, 
\[x\leq y\mbox{ if } (x\in V(G_i), \: y\in V(G_j), \: i<j)\mbox{ or } (x,y\in V(G_i), x\leq y).\]
We deduce a total order on $\sigma(G_1,\ldots,G_k)$ in this way: for any $x,y \in C$,
\[x\leq y \mbox{ if }\min(\sigma^{-1}(x))\leq \min(\sigma^{-1}(y)).\]
\end{remark}

\begin{example}
If $G$ and $H$ are the Feynman graphs of Example \ref{ex1} and:
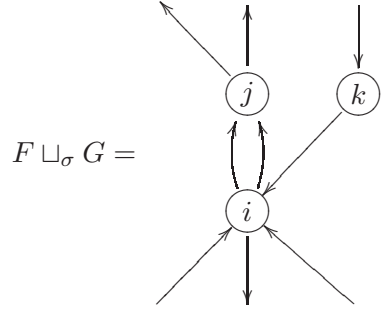
\[\sigma:\left\{\begin{array}{rcl}
a&\mapsto&i\\
b&\mapsto&j\\
c&\mapsto&i\\
d&\mapsto&k\\
\end{array}\right.\]
then:
\[F\sqcup_\sigma G=\mbox{\parbox{3mm}{\xymatrix{&&\ar[d]\\
&\rond{j}\ar[u]\ar[lu]&\rond{k}\ar[ld]\\
&\rond{i}\ar@/^/[u]\ar@/_/[u]\ar[d]&&\\
\ar[ru]&&\ar[lu]}}}\]
\end{example}

We shall use the following particular case:

\begin{defi}\label{defi20}
Let $G$ and $H$ be Feynman graphs, $A\subseteq V(G)$ and $\sigma:A\longrightarrow V(H)$ an injection.
We define an equivalence $\sim$ on $V(G)\sqcup V(H)$ by:
\[\forall a,b \in V(G)\sqcup V(H),\: a\sim_\sigma b \mbox{ if }(a=b)\mbox{ or }(a\in A \mbox{ and }b=\sigma(a)) \mbox{ or }(b\in A \mbox{ and }a=\sigma(b)).\]
Let $\pi$ be the canonical surjection from $V(G) \sqcup V(H)$ to $(V(G)\sqcup V(H))/\sim$.
The Feynman graph $\pi(G,H)$ is denoted by $G\sqcup_\sigma H$. 
\end{defi}

In particular, if $A=\emptyset$, $G\sqcup_\sigma H$ is the disjoint union $GH$.

\subsection{Monomials and Feynman graphs}

\begin{defi}
\begin{enumerate}
\item Let $M$ be a monomial of $A_q(X)$:
\[M=\prod_{i\in X} x_i^{\epsilon_i} \prod_{i\leq_X j} x_{i,j}^{\alpha_{i,j}}\prod_{j\in X} x_{-\infty,j}^{\beta_j}
\prod_{i\in X} x_{i,+\infty}^{\gamma_i}.\]
\begin{enumerate}
\item We shall say that $M$ is \emph{admissible} if:
\begin{align*}
&\forall i,j\in X,&(\alpha_{i,j}\geq 1)&\Longrightarrow (\epsilon_i=\epsilon_j=1),\\
&\forall i\in X,&(\alpha_{i,\infty}\geq 1)&\Longrightarrow (\epsilon_i=1),\\
&\forall j\in X,&(\alpha_{-\infty,j}\geq 1)&\Longrightarrow (\epsilon_j=1).
\end{align*}
\item If $M$ is admissible, we attach to $M$ a Feynman graph $fg(M)$, defined in this way:
\begin{itemize}
\item The set of vertices of $fg(M)$ is the set of elements $i\in X$, such that $\epsilon_i=1$.
\item The number of internal edges $\xymatrix{\rond{i}\ar[r]&\rond{j}}$ in $fg(M)$ is $\alpha_{i,j}$.
\item The number of external edges $\xymatrix{\ar[r]&\rond{j}}$ in $fg(M)$ is $\beta_j$.
\item The number of external edges $\xymatrix{\rond{i}\ar[r]&}$ in $fg(M)$ is $\gamma_i$.
\end{itemize} \end{enumerate}
\item Let $M$ be a monomial of $\bfA_q(X)$:
\[M=x_{i_1}\ldots x_{i_k} \prod_{i\leq_X j} x_{i,j}^{\alpha_{i,j}}\prod_{j\in X} x_{-\infty,j}^{\beta_j}
\prod_{i\in X} x_{i,+\infty}^{\gamma_i}.\]
We shall say that $M$ is admissible if its image $\overline{M}$ in $A_q(X)$ (which is a monomial) is admissible.
The Feynman graph  $fg(\overline{M})$ attached to the image of $M$ in the quotient $A_q(X)$ of $\bfA_q(M)$ is ordered:
the set of its vertices is $\{i_1,\ldots,i_k\}$, totally ordered by $i_1<\ldots <i_k$.
This ordered Feynman graph is denoted by $fg(M)$.
\end{enumerate}\end{defi}

\begin{example} 
Let $a,b\in X$, $a\leq_X b$, $a\neq b$. In $\bfA_q(X)$:
\begin{align*}
fg(x_ax_b x_{-\infty,a}x_{a,b}^2x_{b,\infty}^3)&=
\mbox{\parbox{3mm}{\xymatrix{&&\\
&\rond{2}\ar[lu]\ar[u]\ar[ru]&\\
&\rond{1}\ar@/^/[u]\ar@/_/[u]&\\
&\ar[u]&}}},&
fg(x_bx_a x_{-\infty,a}x_{a,b}^2x_{b,\infty}^3)&=
\mbox{\parbox{3mm}{\xymatrix{&&\\
&\rond{1}\ar[lu]\ar[u]\ar[ru]&\\
&\rond{2}\ar@/^/[u]\ar@/_/[u]&\\
&\ar[u]&}}}.
\end{align*}
In $A_q(X)$:
\begin{align*}
fg(x_ax_b x_{-\infty,a}x_{a,b}^2x_{b,\infty}^3)&=
\mbox{\parbox{3mm}{\xymatrix{&&\\
&\rond{}\ar[lu]\ar[u]\ar[ru]&\\
&\rond{}\ar@/^/[u]\ar@/_/[u]&\\
&\ar[u]&}}}.
\end{align*}
\end{example}

\begin{lemma}\label{lemme16}
Let $M$ and $N$ be two admissible monomials in $A_q(X)$ or in $\bfA_q(X)$. We put $A=V(fg(M))\cap V(fg(N))$ and $\sigma:A\longrightarrow V(fg(N))$ 
be the canonical injection. There exists an admissible monomial $P$, 
such that in $A_q(X)$:
\[MN=q^{|A|} P.\]
Moreover:
\[fg(P)=\pi_\sigma(fg(M),fg(N)).\]
\end{lemma}

\begin{proof} We work in $A_q(X)$; the proof for $\bfA_q(X)$ is similar. We put:
\begin{align*}
M&=\prod_{i\in B} x_i \prod_{i\leq_X j} x_{i,j}^{\alpha_{i,j}}\prod_{j\in X} x_{-\infty,j}^{\beta_j}
\prod_{i\in X} x_{i,+\infty}^{\gamma_i},\\
M&=\prod_{i\in C} x_i \prod_{i\leq_X j} x_{i,j}^{\alpha'_{i,j}}\prod_{j\in X} x_{-\infty,j}^{\beta'_j}
\prod_{i\in X} x_{i,+\infty}^{\gamma'_i},
\end{align*}
then, as $x_i^2=qx_i$ in $A_q(X)$:
\begin{align*}
P&=\prod_{i\in B\cup C} x_i^{\epsilon_i} \prod_{i\leq_X j} x_{i,j}^{\alpha_{i,j}+\alpha'_{i,j}}\prod_{j\in X} x_{-\infty,j}^{\beta_j+\beta'_j}
\prod_{i\in X} x_{i,+\infty}^{\gamma_i+\gamma'_i},
\end{align*}
and $MN=q^{|A|} P$. \end{proof}

\begin{theo}\label{theoprod}
\begin{enumerate}
\item Let $G$ be a Feynman graph and let $X$ a quasi-ordered alphabet. We put:
\begin{align*}
M_G(X)&=\sum_{\substack{\mbox{\scriptsize $M$ admissible monomial},\\fg(M)\approx G}} M\in A_q(X),\\
H_{\fg}(X)&=Vect(M_G(X)\mid G\mbox{ Feynman graph}).
\end{align*}
Then $H_{\fg}(X)$ is a subalgebra of $A_q(X)$. For any Feynman graphs $G$ and $H$:
\begin{align*}
M_G(X)M_H(X)&=\sum_{\sigma: V(G)\supseteq A\hookrightarrow V(H)} q^{|A|} M_{G\sqcup_\sigma H}(X).
\end{align*}
Moreover, there exists a quasi-ordered alphabet $X$, such that the elements $M_G(X)$ 
are linearly independent in $A_q(X)$.
\item \begin{enumerate}
\item Let $G$ be an ordered Feynman graph and let $X$ a quasi-ordered alphabet. We put:
\begin{align*}
\bfM_G(X)&=\sum_{\substack{\mbox{\scriptsize $M$ admissible monomial},\\fg(M)\approx G}} M\in \bfA_q(X),\\
\bfH_{\fg}(X)&=Vect(\bfM_G(X)\mid G\mbox{ ordered Feynman graph}).
\end{align*}
Then $\bfH_{\fg}(X)$ is a subalgebra of $\bfA_q(X)$. For any ordered Feynman graphs $G$ and $H$:
\begin{align*}
\bfM_G(X)\bfM_H(X)&=\sum_{\sigma: V(G)\supseteq A\hookrightarrow V(H)} q^{|A|} \bfM_{G\sqcup_\sigma H}(X).
\end{align*}
There exists a quasi-ordered alphabet $X$, such that the elements $\bfM_G(X)$ are linearly independent in $\bfA_q(X)$.
\end{enumerate}
\end{enumerate}
\end{theo}

\begin{proof} 1. By Lemma \ref{lemme16}, $M_G(X)M_H(X)$ is a linear sums of terms $q^{|A|} P$, 
with $fg(P)=\pi_\sigma(F,G)$. It remains to show that all such terms are obtained, 
which is immediate. So $H_{\fg}(X)$ is a subalgebra of $A_q(X)$.\\

Let $X$ be an infinite set. We give it a relation $\leq_X$ by:
\begin{align*}
&\forall x,y\in X,&x\leq_X y,
\end{align*}
making it a quasi-ordered alphabet. Let $G$ be a Feynman graph. For any $i,j\in V(G)$, we define:
\begin{itemize}
\item $\alpha_{i,j}$ is the number of internal edges $\xymatrix{\rond{i}\ar[r]&\rond{j}}$ in $G$.
\item $\beta_j$ is the number of external edges $\xymatrix{\ar[r]&\rond{j}}$ in $G$.
\item $\gamma_i$ is the number of external edges $\xymatrix{\rond{i}\ar[r]&}$ in $G$.
\end{itemize}
Let $\tau:V(G)\longrightarrow X$ be an injection: this exists, as $V(G)$ is finite and $X$ is infinite. We consider:
\[M=\prod_{i\in V(G)} x_{\tau(i)} \prod_{i,j\in V(G)}x_{\tau(i),\tau(j)}^{\alpha_{i,j}} \prod_{j\in V(G)}x_{-\infty,\tau(j)}^{\beta_j} \prod_{i\in V(G)}x_{\tau(i),\infty}^{\gamma_i}.\]
Obviously, $M$ is admissible and $fg(M)\approx G$, so $M_G(X)$ is non-zero.
As the elements of $M_G(X)$ are all non-zero and their support are disjoint, they are linearly independent. \\

2. Similar proof. \end{proof}

\begin{remark}
In particular, if $q=0$:
\begin{align*}
M_G(X)M_H(X)&=M_{GH}(X), &\bfM_G(X)\bfM_H(X)&=\bfM_{G.H}(X).
\end{align*} \end{remark}

\subsection{The first coproduct}

\begin{defi}
Let $G$ be a Feynman graph, and let $A\subseteq V(G)$. 
\begin{enumerate}
\item We define a Feynman graph $G_{\mid A}$ by the following:
\begin{itemize}
\item The set of vertices of $G_{\mid A}$ is $A$.
\item The set of half-edges of $G_{\mid A}$ is the set of half-edges $e\in HE(G)$ such that $s_G(e)\in A$.
\item For all half-edge $e$ of $G_{\mid A}$:
\begin{align*}
type_{G_{\mid A}}(e)&=type_G(e),\\
s_{G_{\mid A}}(e)&=s_G(e),\\
i_{G_{\mid A}}(e)&=\begin{cases}
i_G(e)\mbox{ if }s_G(i_G(e))\in A,\\
e\mbox{ otherwise}.
\end{cases}\end{align*}\end{itemize}
\item We shall say that $A$ is an \emph{ideal} of $G$ if for all $i,j\in V(G)$, if $i\in A$ and if
 there is an edge from $i$ to $j$ in $G$, then $j\in A$. The set of ideals of $G$ is denoted by $I(G)$.
\end{enumerate}\end{defi}

Note that if $G$ is an ordered Feynman graph, then for any $A$, $G_{\mid A}$ is also ordered.

\begin{theo}\label{theocop}
Let $G$ be a Feynman graph. For any quasi-ordered alphabets $X$ and $Y$:
\begin{align*}
\Delta_{X,Y}(M_G(X\sqcup Y))&=\sum_{A\in I(G)}M_{G_{\mid V(G)\setminus A}}(X)\otimes M_{G_{\mid A}}(Y),\\
\Delta_{X,Y}(\bfM_G(X\sqcup Y))&=\sum_{A\in I(G)}\bfM_{G_{\mid V(G)\setminus A}}(X)\otimes \bfM_{G_{\mid A}}(Y).
\end{align*}\end{theo}

\begin{proof} We consider the two following sets:
\begin{itemize}
\item $\mathcal{A}$ is the set of triples $(M,M_1,M_2)$, where $M$, $M_1$ and $M_2$ are monomials of respectively
$A_q(X\sqcup Y)$, $A_q(X)$ and $A_q(Y)$, such that $\Delta_{X,Y}(M)=M_1\otimes M_2$ and $fg(M)=G$.
\item $\mathcal{B}$ is the set of triples $(A,M_1,M_2)$, where $A$ is an ideal of $G$, $M_1$ is an admissible monomial of $A_q(X)$ such
that $fg(M_1)=G_{\mid V(G)\setminus A}$, $M_2$ is an admissible monomial of $A_q(Y)$ such that $fg(M_2)=G_{\mid A}$.
\end{itemize}
Let $(M,M_1,M_2)\in \mathcal{A}$. Let $A$ be the set of vertices of $fg(M)$ belonging to $Y$. By Definition of $\Delta_{X,Y}$,
as $\Delta_{X,Y}(M)=M_1\otimes M_2$, $fg(M_1)=G_{\mid V(G)\setminus A}$ and $fg(M_2)=G_{\mid A}$.
Moreover, if $i\in A$ and if there is an edge from $i$ to $j$ in $G$, this means that $x_{i,j}$ appears in $M$,
so $i\leq_{X\sqcup Y} j$. As $i\in Y$, $j\in Y$, so $j\in A$: $A\in I(G)$. We define in this way a map:
\[\left\{\begin{array}{rcl}
\mathcal{A}&\longrightarrow&\mathcal{B}\\
(M,M_1,M_2)&\longrightarrow&(Vert(fg(M))\cap Y,M_1,M_2)
\end{array}\right.\]
It is not difficult to prove that it is a bijection. Hence:
\begin{align*}
\Delta_{X,Y}(M_G(X\sqcup Y))&=\sum_{(M,M_1,M_2)\in \mathcal{A}} M_1\otimes M_2\\
&=\sum_{(A,M_1,M_2)\in \mathcal{B}}M_1\otimes M_2\\
&=\sum_{A\in I(G)}M_{G_{\mid V(G)\setminus A}}(X)\otimes M_{G_{\mid A}}(Y).
\end{align*}
The proof in $\bfA_q(X\sqcup Y)$ is similar. \end{proof}

\begin{cor}\label{corDelta}
\begin{enumerate}
\item Let $H_{\fg}$ be the vector space generated by the set of Feynman graphs and let $q\in \K$. 
It is given a Hopf algebra structure: for any Feynman graphs $G$ and $H$,
\begin{align}
\label{E3}G._qH&=\sum_{\sigma:V(G)\supseteq A\hookrightarrow V(H)} q^{|A|}G\sqcup_\sigma H\\
\label{E4}\Delta(G)&=\sum_{A\in I(G)}G_{\mid V(G)\setminus A}\otimes G_{\mid A}.
\end{align}
\item Let $\bfH_{\fg}$ be the vector space generated by the set of ordered Feynman graphs and let $q\in \K$. 
It is given a Hopf algebra structure  by (\ref{E3})-(\ref{E4}).
\end{enumerate}\end{cor}

\begin{proof} 1. For any quasi-ordered alphabet $X$, there is linear map:
\[\Theta_X:\left\{\begin{array}{rcl}
H_{\fg}&\longrightarrow&H_{\fg}(X)\\
G&\longrightarrow&M_G(X).
\end{array}\right.\]
By Theorems \ref{theoprod} and \ref{theocop}, for any quasi-ordered alphabets $X$ and $Y$, for any $a,b\in H_{\fg}$:
\begin{align*}
\Theta_X(a)\Theta_X(b)&=\Theta_X(ab),\\
(\Theta_X \otimes \Theta_Y)\circ \Delta(a)&=\Delta_{X,Y}\circ \Theta_{X\sqcup Y}(a).
\end{align*}
Choosing quasi-ordered alphabets $X$, $Y$ and $Z$, such that $\Theta_X$, $\Theta_Y$ and $\Theta_Z$
(Proposition \ref{theoprod}), we obtain that $H_\fg$ is an algebra. For all $a,b\in H_\fg$:
\begin{align*}
(\Theta_X \otimes \Theta_Y)\circ \Delta(a._q b)&=\Delta_{X,Y}\circ \Theta_{X\sqcup Y}(ab)\\
&=\Delta_{X,Y}(\Theta_{X\sqcup Y}(a)\Theta_{X\sqcup Y}(b))\\
&=\Delta_{X,Y}\circ \Theta_{X\sqcup Y}(a)\Delta_{X,Y} \circ \Theta_{X\sqcup Y}(b)\\
&=(\Theta_X\otimes \Theta_Y)\circ \Delta(a)(\Theta_X\otimes \Theta_Y)\circ \Delta(b)\\
&=(\Theta_X \otimes \Theta_Y)(\Delta(a)\Delta(b)).
\end{align*}
As $\Theta_X$ and $\Theta_Y$ are injective, $\Theta_X\otimes \Theta_Y$ is too, so $\Delta(ab)=\Delta(a)\Delta(b)$. Moreover:
\begin{align*}
(\Theta_X \otimes \Theta_Y \otimes \Theta_Z)\circ (\Delta\otimes Id)\circ \Delta
&=(\Delta_{X,Y}\otimes Id)\circ \Delta_{X\sqcup Y,Z}\circ \Theta_{X\sqcup Y\sqcup Z}\\
&=(Id \otimes \Delta_{Y,Z})\circ \Delta_{X,Y\sqcup Z}\circ \Theta_{X\sqcup Y\sqcup Z}\\
&=(\Theta_X \otimes \Theta_Y \otimes \Theta_Z)\circ (Id \otimes \Delta)\circ \Delta.
\end{align*}
By the injectivity of $\Theta_X \otimes \Theta_Y \otimes \Theta_Z$, $\Delta$ is coassociative, so $H_{\fg}$ is a bialgebra.\\

Observe that $(H_{\fg},m,\Delta)$ is filtered by the cardinality of the set of vertices of Feynman graphs, even graded if $q=0$;
as its components of degree $0$ is reduced to $\K$, it is connected. Hence, it is a Hopf algebra. \\

2. Similar proof. \end{proof}

\subsection{The second coproduct}

\begin{defi}\label{defiequi}
Let $G$ be a Feynman graph, and let $\sim$ be an equivalence on $V(G)$.
\begin{enumerate}
\item We shall say that $\sim$ is  $G$-compatible if the following assertions hold:
\begin{itemize}
\item If $A\subseteq V(G)$ is an equivalence class of $\sim$, then $G_{\mid A}$ is a connected Feynman graph.
\item If there is a path $x_1\rightarrow x_2 \rightarrow\ldots \rightarrow x_k$ in $G$, with $x_1\sim x_k$, then $x_1\sim x_2\sim \ldots \sim x_k$.
\end{itemize}
The set of $G$-compatible equivalences  will be denoted by $CE(G)$. 
\item If $A_1,\ldots,A_k$ are the equivalence classes of $\sim$, we denote by $G_{\mid \sim}$ the disjoint union of the Feynman graphs
$G_{\mid A_i}$, $1\leq i\leq k$.
\item We denote by $G/\sim$ the following Feynman graph:
\begin{itemize}
\item The set of vertices of $G/\sim$ is $V(G)$.
\item The half-edges of $G/\sim$ are the half-edges of $G$ such that $i(e)=e$ or not $s(e)\sim s(i(e))$.
\item For any half-edge $e$ of $G/\sim$:
\begin{align*}
type_{G/\sim}(e)&=type_G(e),\\
s_{G/\sim}(e)&=s_G(e),\\
i_{G/\sim}(e)&=\begin{cases}
i_G(e)\mbox{ if not } s(e)\sim s(i(e)),\\
0\mbox{ otherwise}.
\end{cases}
\end{align*}\end{itemize}\end{enumerate}\end{defi}

Roughly speaking, $G_{\mid \sim}$ is obtained by the deletion in $G$ of all the internal edges whose two extremities are not $\sim$-equivalent,
whereas $G/\sim$ is obtained the deletion in $G$ of all the internal edges whose two extremities are $\sim$-equivalent.

\begin{remark} \begin{enumerate}
\item If $G$ is ordered, then, as $V(G_{\mid \sim})=V(G/\sim)=V(G)$, the Feynman graphs $G_{\mid \sim}$ and $G/\sim$ are also ordered.
\item The Feynman graph $G/\sim$ is not the usual contraction of $G$ according to a subgraph used by Connes and Kreimer
to define a coproduct on Feynman graphs \cite{CK1,CK2,CK3,CK4,CK5,CK6,Manchon}, as here all the vertices are conserved.
\end{enumerate} \end{remark}

\begin{notation} Let $q\in \K$. For any Feynman graph $G$, if $G_1,\ldots,G_k$ are its connected components, we put:
\[\psi_q(G)=G_1\cdot_q \ldots \cdots_q G_k\in A_q(G).\]
In particular, $\psi_0(G)=G$.
\end{notation}

\begin{prop} Let $X$, $Y$ be two quasi-ordered alphabets. 
\begin{enumerate}
\item We assume that $q=q_1q_2$. We consider $\delta_{X,Y}:A_q(XY)\longrightarrow A_{q_1}(X)\otimes A_{q_2}(Y)$. For any Feynman graph $G$:
\begin{align*}
\delta_{X,Y}(M_G(XY))&=\sum_{\sim \in CE(G)} M_{\psi_{q_1}(G/\sim)}(X)\otimes M_{\psi_{q_2}(G_{\mid \sim})}(Y).
\end{align*}
\item We consider $\delta_{X,Y}:\bfA_0(XY)\longrightarrow \bfA_0(X)\otimes \bfA_0(Y)$.
For any ordered Feynman graph $G$, in $\bfA_0(X)$:
\begin{align*}
\delta_{X,Y}(\bfM_G(XY))&=\sum_{\sim \in CE(G)} \bfM_{G/\sim}(X)\otimes \bfM_{G_{\mid \sim}}(Y).
\end{align*}
\end{enumerate} \end{prop}

\begin{proof} 1. Let $M$ be a monomial of $A_q(XY)$, such that $fg(M)=G$. The set of vertices of $fg(M)$ is denoted by 
$V(G)=\{(i_1,j_1),\ldots, (i_k,j_k)\}$. We define an equivalence on $V(G)$ by:
\[(i_p,j_p)\sim (i_q,j_q)\mbox{ if }i_p\sim_X i_q.\]
Let us assume that there is a path $(i_{p_1},j_{p_1})\rightarrow\ldots \rightarrow (i_{p_k},j_{p_k})$ in $G$, with $i_{p_1}\sim_X j_{p_1}$.
Then, as $G=fg(M)$, in $XY$:
\[(i_{p_1},j_{p_1})\leq_{XY}\ldots \leq_{XY}(i_{p_k},j_{p_k}),\]
so, in $X$:
\[i_{p_1}\leq_X \ldots \leq_X i_{p_k}.\]
As $i_{p_1}\sim_X i_{p_k}$, we obtain that $i_{p_1}\sim_X \ldots \sim_X i_{p_k}$, and finally:
\[(i_{p_1},j_{p_1})\sim \ldots \sim(i_{p_k},j_{p_k}).\]
Moreover, $\delta_{X,Y}(M)=M_1\otimes M_2$, with, by Definition of $\delta_{X,Y}$, $fg(M_1)=G/\sim$ and $fg(M_2)=G_{\mid \sim}$.
Let $\sim'$ be the equivalence which equivalent classes are the connected components of $G_{\mid \sim}$.
Then $\sim' \in CE(G)$ and:
\begin{align}
\label{EQ1}G_{\mid \sim}&=G_{\mid \sim'},&G/\sim&=G/\sim'.
\end{align}
Moreover, $\sim'$ is the unique element of $CE(G)$ such that (\ref{EQ1}) holds. Finally, we proved that $\delta_{X,Y}$ is a sum of terms
$M_1\otimes M_2$, such that there exists $\sim'\in CE(G)$, such that $fg(M_1)$ is a monomial of $\psi_{q_1}(G/\sim')$ and 
$fg(M_2)$ is a monomial of $\psi_{q_2}(G_{\mid \sim'})$. It is not difficult to see
that all these terms are obtained. \\

2. Similar proof. \end{proof}

\begin{cor}
\begin{enumerate}
\item We define  two coproducts on $H_{\fg}$ for any Feynman graph $G$ by:
\begin{align}
\label{E5}\delta(G)&=\sum_{\sim \in CE(G)}  (G/\sim)\otimes G_{\mid \sim},\\
\nonumber \delta_1(G)&=\sum_{\sim \in CE(G)}  \psi_1(G/\sim)\otimes \psi_1(G_{\mid \sim}).
\end{align}
Then $(H_{\fg},._0,\delta)$ and $(H_{\fg},._1,\delta_1)$ are bialgebras. For any $q_1,q_2 \in \K$, 
we define a coaction by:
\begin{align*}
\rho_{q_1,q_2}(G)&=\sum_{\sim \in CE(G)}  \psi_{q_1}(G/\sim)\otimes \psi_{q_2}(G_{\mid \sim}).
\end{align*}
If $(q_1,q_2)=(q,1)$ with $q\in \K$, or $(0,0)$, then $(H_{\fg},._{q_1},\Delta)$ is a bialgebra in the category 
of right $(H_{\fg},._{q_2},\delta_{q_2})$-comodules by the coaction $\rho_{q_1,q_2}$, that is to say:
\begin{itemize}
\item For all $x,y \in H_{\fg}$, $\rho_{q_1,q_2}(x._{q_1} y)=\rho_{q_1,q_2}(x)._{q_1,q_2} \rho_{q_1,q_2}(y)$, where:
\[._{q_1,q_2}:\left\{\begin{array}{rcl}
H_{\fg}^{\otimes 4}&\longrightarrow&H_{\fg}^{\otimes 2}\\
x\otimes y\otimes z\otimes t&\longrightarrow&x._{q_1} z\otimes y._{q_2} z.
\end{array}\right.\]
\item $\rho_{q_1,q_2}(1)=1\otimes 1$.
\item $(\Delta \otimes Id)\circ \rho_{q_1,q_2}=
m_{1,3,24}\circ (\rho_{q_1,q_2} \otimes \rho_{q_1,q_2}) \circ \Delta$, where:
\[m_{1,3,24}:\left\{\begin{array}{rcl}
H_{\fg}^{\otimes 4}&\longrightarrow&H_{\fg}^{\otimes 3}\\
x\otimes y\otimes z\otimes t&\longrightarrow&x\otimes z\otimes y._{q_2} t.
\end{array}\right.\]
\item For all $x\in H_{\fg}$, $(\varepsilon_\Delta \otimes Id)\circ \rho_{q_1,q_2}(x)=\varepsilon_\Delta(x)1$.
\end{itemize}
Note that $\rho_{0,0}=\delta$. 
\item Similarly, (\ref{E5}) defines a coproduct $\delta$ on $\bfA_q(0)$, making it a bialgebra on the category of right $(H_{\fg},._0,\delta)$-comodules.
\end{enumerate}

makes 
of $(H_{\fg},._0,\delta)$-comodules, \end{cor}

\begin{proof} We take $(q_1,q_2)=(q,1)$ or $(0,0)$. Hence, $q_1q_2=q_1$. 
For any quasi-ordered alphabets $X$, $Y$, $(\Theta_X\otimes \Theta_Y)\circ \delta=\delta_{X,Y}\circ \Theta_{XY}$.
The proof that $(H_{\fg},._0,\delta)$, $(\bfH_{\fg},._0,\delta)$, $(H_{\fg},._1,\delta_1)$ 
and $(\bfH_{\fg},._1,\delta_1)$ are bialgebras is similar to the proof of corollary \ref{corDelta}.
The only non-trivial remaining assertion to prove is point 3. Let us take $X,Y,Z$ alphabets such that $\Theta_X$, $\Theta_Y$ and $\Theta_Z$ are injective.
\begin{align*}
(\Theta_X\otimes \Theta_Y\otimes \Theta_Z)\circ (\Delta \otimes Id)\circ \delta
&=(\Delta_{X,Y}\otimes Id)\circ \delta_{X\sqcup Y,Z}\circ \Theta_{XYZ}\\
&=m_{1,3,24}\circ (\delta_{X,Z} \otimes \delta_{Y,Z}) \circ \Delta_{XZ,YZ} \circ \Theta_{XYZ}\\
&=(\Theta_X\otimes \Theta_Y\otimes \Theta_Z) \circ m_{1,3,24}\circ (\delta \otimes \delta) \circ \Delta.
\end{align*}
We conclude by the injectivity of $\Theta_X\otimes \Theta_Y\otimes \Theta_Z$. \end{proof}

\subsection{Restriction to ordered alphabets}

Let $G$ be a Feynman graph. A \emph{cycle} in $G$ is a sequence of vertices $(i_1,\ldots,i_k)$, with $k\geq 2$, all distinct,
such that there exist internal edges $e_1,\ldots,e_k$, with:
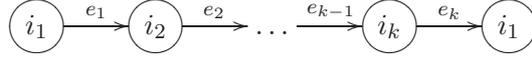
\[\xymatrix{\rond{i_1}\ar[r]^{e_1}&\rond{i_2}\ar[r]^{e_2}&\ldots \ar[r]^{e_{k-1}}&\rond{i_k}\ar[r]^{e_k}&\rond{i_1}}\]

\begin{lemma}\label{graphessanscycles}
Let $G$ be a Feynman graph. The following conditions are equivalent:
\begin{enumerate}
\item For any ordered alphabet, $M_G(X)=0$.
\item $G$ has a cycle.
\end{enumerate}
Moreover, there exists an ordered alphabet $X$, such that for any Feynman graph with no cycle, $M_G(X)\neq 0$.
\end{lemma}

\begin{proof}
$2.\Longrightarrow 1$. Let us assume that $G$ has a cycle $(i_1,\ldots,i_k)$, with $M_G(X)\neq 0$.
There exists $j_1,\ldots,k_k \in X$, all distinct, such that 
$x_{j_1,j_2}\ldots x_{j_{k-1},j_k}x_{j_k,j_1}$ appears in $M_G(X)$, so $j_1\leq_X\ldots \leq_X j_{k-1}\leq_X j_k\leq_X j_1$.
As $\leq_X$ is an order, $j_1=\ldots=j_k$, which is a contradiction. So $M_G(X)=0$.

$1.\Longrightarrow 2$. Let us prove that there exists a total order $\leq$ on $V(G)$, such that if there is an edge from $i$ to $j$
in $G$, then $i\leq j$. We proceed by induction on $|V(G)|$. If $|V(G)|=1$, this is obvious. If $|V(G)|\geq 2$, as $G$ has no cycle,
it has a source $v_1$, that is to say a vertex with no internal incoming edge. The Feynman graph $G'$ obtained from $G$ by deleting $v_0$
and all the attached half-edges has no cycle, so, by the induction hypothesis, the set of its vertices inherits a total order 
$v_2\leq \ldots \leq v_k$, compatible with the internal edges of $G'$. We give $V(G)$ a total order by $v_1\leq v_2\leq \ldots \leq v_k$; 
as $v_1$ is a source, this order is compatible with the internal edges of $G$.\\

We take $X=\N$, with its usual order. There exists a monomial of $A_q(X)$ of the form:
\[M=x_1\ldots x_k\prod_{1\leq i\leq j\leq n} x_{i,j}^{\alpha_{i,j}} \prod_{1\leq i\leq n} x_{i,\infty}^{\beta_i} \prod_{1\leq j\leq n}x_{-\infty,j}^{\gamma_j},\]
such that $fg(M)=G$: $\alpha_{i,j}$ is the number of internal edges between between $v_i$ and $v_j$, $\beta_i$ is the number of incoming 
external edges in $v_i$, and $\gamma_j$ is the number of external half-edges outgoing from $v_j$, if $v_1\leq \ldots \leq v_k$
is the previously defined order on the set of vertices of $G$. So $M_G(X)\neq 0$. \end{proof} 

Observe that if $G$ has no cycle and $\sim\in CE(G)$, then both $G_{\mid \sim}$ and $G/\sim$
has no cycle. Hence, considering only ordered alphabets, we obtain a Hopf algebra and a bialgebra of Feynman graphs with no cycle:

\begin{theo} \label{theofgwc}
Let $H_{\ncfg}$ be the vector space generated by the set of Feynman graphs with no cycle and let $q\in \K$. 
\begin{enumerate}
\item It is given a Hopf algebra structure: for any Feynman graphs $G$ and $H$ with no cycle,
\begin{align}
\label{E6} G._qH&=\sum_{\substack{\sigma:V(G)\supseteq A\hookrightarrow V(H),\\ \mbox{\scriptsize $G\sqcup_\sigma H$ with no cycle}}}
q^{|A|} G\sqcup_\sigma H,\\
\label{E7}\Delta(G)&=\sum_{A\in I(G)}G_{\mid V(G)\setminus A}\otimes G_{\mid A}.
\end{align}
\item We define a second coproduct by:
\begin{align}
\label{E8}\delta(G)&=\sum_{\sim \in CE(G)}  (G/\sim)\otimes (G_{\mid \sim}).
\end{align}
Then $(H_{\ncfg},._0,\delta)$ is a  bialgebra. Moreover, $(H_{\ncfg},._0,\Delta)$ is a bialgebra in the category
of $(H_{\ncfg},._0,\delta)$-comodules.
\item Let us consider the map:
\[T:\left\{\begin{array}{rcl}
H_\fg&\longrightarrow&H_\ncfg\\
G&\longrightarrow&\begin{cases}
G\mbox{ if $G$ has no cycle},\\
0\mbox{ otherwise}.
\end{cases}
\end{array}\right.\]
It is a Hopf algebra morphism from $(H_\fg,._q,\Delta)$ to $(H_\ncfg,._q,\Delta)$ and from $(H_\fg,._0,\delta)$ to
$(H_\ncfg,._0,\delta)$.
\end{enumerate}\end{theo}

Here is its non-commutative version:

\begin{theo}
Let $\bfH_{\ncfg}$ be the vector space generated by the set of ordered Feynman graphs with no cycle and let $q\in \K$. 
\begin{enumerate}
\item It is given a Hopf algebra structure by (\ref{E6})-(\ref{E7}).
\item We define a second coproduct by (\ref{E8}).
Then $(\bfH_{\ncfg},._0,\delta)$ is a  bialgebra. Moreover, $(\bfH_{\ncfg},._0,\Delta)$ is a bialgebra in the category
of $(H_{\ncfg},._0,\delta)$-comodules.
\item Let us consider the map:
\[T:\left\{\begin{array}{rcl}
\bfH_\fg&\longrightarrow&\bfH_\ncfg\\
G&\longrightarrow&\begin{cases}
G\mbox{ if $G$ has no cycle},\\
0\mbox{ otherwise}.
\end{cases}
\end{array}\right.\]
It is a Hopf algebra morphism from $(\bfH_\fg,._q,\Delta)$ to $(\bfH_\ncfg,._q,\Delta)$ and from $(\bfH_\fg,._0,\delta)$ to
$(\bfH_\ncfg,._0,\delta)$.
\end{enumerate}\end{theo}

\begin{remark}
There is a canonical injection $\iota : H_\ncfg\longrightarrow H_\fg$. This injection $\iota$ is compatible with $._q$ if, and only if,
$q=0$: indeed, as $._0$ is the disjoint union, it is compatible with $\iota$. If $q\neq 0$, taking $G=H=\xymatrix{\rond{}\ar[r]&\rond{}}$:
\begin{itemize}
\item In $H_\fg$:
\begin{align*}
G._q H&=\xymatrix{\rond{}\ar[r]&\rond{}}\:\xymatrix{\rond{}\ar[r]&\rond{}}+q\:\xymatrix{\rond{}&\rond{}\ar[r]\ar[l]&\rond{}}
+2q\:\xymatrix{\rond{}\ar[r]&\rond{}\ar[r]&\rond{}}\\[2mm]
&+q\:\xymatrix{\rond{}\ar[r]&\rond{}&\rond{}\ar[l]}+q^2\:\xymatrix{\rond{}\ar@/_1pc/[r]\ar@/^1pc/[r]&\rond{}}
+q^2\:\xymatrix{\rond{}\ar@/_1pc/[r]&\rond{}\ar@/_1pc/[l]}.
\end{align*}
\item In $H_\ncfg$:
\begin{align*}
G._q H&=\xymatrix{\rond{}\ar[r]&\rond{}}\:\xymatrix{\rond{}\ar[r]&\rond{}}+q\:\xymatrix{\rond{}&\rond{}\ar[r]\ar[l]&\rond{}}
+2q\:\xymatrix{\rond{}\ar[r]&\rond{}\ar[r]&\rond{}}\\[2mm]
&+q\:\xymatrix{\rond{}\ar[r]&\rond{}&\rond{}\ar[l]}+q^2\:\xymatrix{\rond{}\ar@/_1pc/[r]\ar@/^1pc/[r]&\rond{}}.
\end{align*}
\end{itemize}
However, $\iota$ is compatible with $\Delta$ and $\delta$.
\end{remark}

\section{Quotients of Feynman graphs}

\subsection{Simple oriented graphs}

Let $X$ be a quasi-ordered alphabet. We consider the following quotients of $A_q(X)$ and $\bfA_q(X)$:
\begin{align*}
A'_q(X)&=\frac{A_q(X)}{\langle x_{i,\infty}-1,x_{-\infty,j}-1, x_{i,j}^2-x_{i,j},x_{i,i}-1, i,j\in X\rangle},\\
\bfA'_q(X)&=\frac{\bfA_q(X)}{\langle x_{i,\infty}-1,x_{-\infty,j}-1, x_{i,j}^2-x_{i,j},x_{i,i}-1, i,j\in X\rangle}.
\end{align*}
Elements of $A'_q[X]$ are formal spans of monomials
\[M=\prod_{i\in X} x_i^{\epsilon_i} \prod_{i\leq_X j} x_{i,j}^{\alpha_{i,j}},\]
where the $\epsilon_i,\alpha_{i,j}\in\{0,1\}$, with only a finite number of them non-zero. Elements of $\bfA'_q[X]$ are formal spans of monomials
\[M=x_{i_1}\ldots x_{i_k} \prod_{i\leq_X j} x_{i,j}^{\alpha_{i,j}},\]
where $i_1,\ldots,i_k$ are elements of $X$, all distinct, $\alpha_{i,j}\in \{0,1\}$, with only a finite number of them non-zero.
The canonical surjections from $A_q(X)$ to $A'_q(X)$ or form $\bfA_q(X)$ to $\bfA'_q(X)$ are both denoted by $\varpi'_X$.

\begin{prop}
Let $X,Y$ be two quasi-ordered alphabets, and $q,q_1,q_2 \in \K$, such that $q=q_1q_2$.
\begin{enumerate}
\item There exist unique algebra morphisms $\Delta_{X,Y}:A'_q(X\sqcup Y)\longrightarrow A'_q(X)\otimes A'_q(Y)$
and $\delta_{X,Y}:A'_q(XY)\longrightarrow A'_{q_1}(X)\otimes A'_{q_2}(Y)$, such that the following diagrams commute:
\begin{align*}
&\xymatrix{A_q(X\sqcup Y)\ar[d]_{\varpi'_{X\sqcup Y}} \ar[r]^{\Delta_{X,Y}}&A_q(X)\otimes A_q(Y)\ar[d]^{\varpi'_X\otimes \varpi'_Y}\\
A'_q(X\sqcup Y)\ar[r]_{\Delta_{X,Y}}&A'_q(X)\otimes A'_q(Y)}&
&\xymatrix{A_q(XY)\ar[d]_{\varpi'_{XY}} \ar[r]^{\delta_{X,Y}}&A_{q_1}(X)\otimes A_{q_2}(Y)\ar[d]^{\varpi'_X\otimes \varpi'_Y}\\
A'_q(XY)\ar[r]_{\delta_{X,Y}}&A'_{q_1}(X)\otimes A'_{q_2}(Y)}
\end{align*}
They are given in the following way: if $i,i'\in X$, $j,j'\in Y$,
\begin{align*}
\Delta_{X,Y}(x_i)&=x_i\otimes 1,&\Delta_{X,Y}(x_{i,i'})&=x_{i,i'}\otimes 1,\\
\Delta_{X,Y}(x_j)&=1\otimes x_j,&\Delta_{X,Y}(x_{j,j'})&=1\otimes x_{j,j'},\\
&&\Delta_{X,Y}(x_{i,j})&=1\otimes 1,&\\ \\
\delta_{X,Y}(x_{(i,j)})&=x_i\otimes x_j,&\delta_{X,Y}(x_{(i,j),(i',j')})&=\begin{cases}
x_{i,i'}\otimes 1\mbox{ if }i<_X i',\\
1\otimes x_{j,j'}\mbox{ if }i\sim_X i'.
\end{cases}
\end{align*}
\item The same assertions hold if one replaces $A_q$ and $A'_q$ by $\bfA_q$ and $\bfA'_q$ everywhere.
\end{enumerate}
\end{prop}

\begin{proof} Immediate verifications. \end{proof}

\begin{defi}
Let $G$ be a Feynman graph. We denote by $S(G)$ the simple oriented graph obtained by the following procedure:
\begin{enumerate}
\item Delete all the external edges of $G$.
\item Delete the loops, that is to say internal edges with two identical extremities.
\item If $G$ has several edges from $i$ to $j$, where $i,j\in V(G)$, keep only one edge from $i$ to $j$ in $S(G)$.
\end{enumerate}\end{defi}

\begin{lemma}
Let $G$ and $H$ be two Feynman graphs. The following conditions are equivalent:
\begin{enumerate}
\item $\varpi'_X(M_G(X))=\varpi'_X(M_H(X))$ for any quasi-ordered alphabet $X$.
\item $S(G)=S(H)$.
\end{enumerate}\end{lemma}

\begin{proof} $2.\Longrightarrow 1$. Let $X$ be a quasi-ordered alphabet. 
As $S(G)=S(H)$, we can assume that $V(G)=V(H)$.
We denote by $E$ the set of pairs $(i,j)$ such that there is an edge from $i$ to $j$ in $G$ and in $H$, with $i\neq j$. 
There exist non-zero scalars $\alpha_{i,j}$, $\alpha'_{i,j}$, for $(i,j)\in E$, and scalars $\alpha_i,$, $\alpha'_i$, 
$\beta_i$, $\beta'_i$, $\gamma_j$, $\gamma'_j$, and a set of injections $\Lambda$ from
$V(G)$ to $X$ such that:
\begin{align*}
M_G(X)&=\sum_{\sigma\in \Lambda} \prod_{i\in V(G)} x_{\sigma(i)} \prod_{(i,j)\in E} x_{\sigma(i),\sigma(j)}^{\alpha_{i,j}}
\prod_{i\in V(G)} x_{\sigma(i),\sigma(i)}^{\alpha_i} \prod_{i\in V(G)} x_{\sigma(i),\infty}^{\beta_i} \prod_{i\in V(G)}x_{-\infty,\sigma(j)}^{\gamma_j},\\
M_H(X)&=\sum_{\sigma\in \Lambda} \prod_{i\in V(G)} x_{\sigma(i)} \prod_{(i,j)\in E} x_{\sigma(i),\sigma(j)}^{\alpha'_{i,j}}
\prod_{i\in V(G)} x_{\sigma(i),\sigma(i)}^{\alpha'_i} \prod_{i\in V(G)} x_{\sigma(i),\infty}^{\beta'_i} \prod_{i\in V(G)}x_{-\infty,\sigma(j)}^{\gamma'_j}.
\end{align*}
Their image under $\varpi'_X$ are both equal to:
\begin{align*}
&\sum_{\sigma\in \Lambda} \prod_{i\in V(G)} x_{\sigma(i)} \prod_{(i,j)\in E} x_{\sigma(i),\sigma(j)}.
\end{align*}

$1.\Longrightarrow 2$. Let us choose an alphabet $X$, such that $M_G(X)$ and $M_H(X)$ are non-zero. Let $M$ be a monomial of $M_G(X)$:
\[M=\prod_{i\in X} x_i^{\epsilon_i} \prod_{i\leq_X j} x_{i,j}^{\alpha_{i,j}}\prod_{j\in X} x_{-\infty,j}^{\beta_j}
\prod_{i\in X} x_{i,+\infty}^{\gamma_i}.\]
Then:
\[\varpi'_X(M)=\prod_{i\in X} x_i^{\epsilon_i} \prod_{i\leq_X j,\: i\neq j} x_{i,j}^{\tilde{\alpha}_{i,j}},\]
where $\tilde{\alpha}'_{i,j}=1$ if $\alpha_{i,j}\neq 0$ and $0$ otherwise. This monomial appears in $M_H(X)$, so there is a monomial $M'$
in $M_H(X)$, of the form:
\[M'=\prod_{i\in X} x_i^{\epsilon_i} \prod_{i\leq_X j} x_{i,j}^{\alpha'_{i,j}}\prod_{j\in X} x_{-\infty,j}^{\beta'_j}
\prod_{i\in X} x_{i,+\infty}^{\gamma'_i},\]
with, if $i\neq j$, $\alpha'_{i,j}\neq 0$ if, and only if, $\tilde{\alpha}_{i,j}=1$ if, and only if, $\alpha_{i,j}\neq 0$.
This implies that $S(G)=S(H)$. \end{proof}

For any simple graph $G$, we denote $M_G(X)=\varpi'_X(M_H(X))$ and $\bfM_G(X)=\varpi'_X(\bfM_H(X))$, 
where $H$ is any Feynman graph such that $S(H)=G$ (for example, $H=G$). These elements, if they are all non-zero, 
form a basis of a subalgebra of $A'_q(X)$ or $\bfA'_q(X)$.
We obtain a quotient of $H_\fg$ and $\bfH_\fg$ based on simple graphs. To sum up:

\begin{theo} \label{theographes}
Let $H_{\sg}$ be the vector space generated by the set of simple graphs and let $q\in \K$. 
\begin{enumerate}
\item It is given a Hopf algebra structure: for any simple graphs $G$ and $H$,
\begin{align}
\label{E9}G._qH&=\sum_{\sigma:V(G)\supseteq A\hookrightarrow V(H)} q^{|A|}S(G\sqcup_\sigma H),\\
\label{E10}\Delta(G)&=\sum_{A\in I(G)}S(G_{\mid V(G)\setminus A})\otimes S(G_{\mid A}).
\end{align}
\item We define a second coproduct by:
\begin{align}
\label{E11}\delta(G)&=\sum_{\sim \in CE(G)}  S(G/\sim)\otimes S(G_{\mid \sim}).
\end{align}
Then $(H_{\sg},._0,\delta)$ is a  bialgebra. Moreover, $(H_{\sg},._0,\Delta)$ is a bialgebra in the category
of $(H_{\sg},._0,\delta)$-comodules.
\item Let us consider the map:
\[S:\left\{\begin{array}{rcl}
H_\fg&\longrightarrow&H_\sg\\
G&\longrightarrow&S(G).
\end{array}\right.\]
It is a Hopf algebra morphism from $(H_\fg,._q,\Delta)$ to $(H_\sg,._q,\Delta)$ and from $(H_\fg,._0,\delta)$ to
$(H_\sg,._0,\delta)$.
\end{enumerate}\end{theo}

It has a non-commutative version:

\begin{theo}
Let $\bfH_{\sg}$ be the vector space generated by the set of ordered simple graphs and let $q\in \K$. 
\begin{enumerate}
\item It is given a Hopf algebra structure by (\ref{E9})-(\ref{E10}).
\item We define a second coproduct by (\ref{E11}).
Then $(\bfH_{\sg},._0,\delta)$ is a bialgebra. Moreover, $(\bfH_{\sg},._0,\Delta)$ is a bialgebra in the category of $(H_{\sg},._0,\delta)$-comodules.
\item Let us consider the map:
\[S:\left\{\begin{array}{rcl}
\bfH_\fg&\longrightarrow&\bfH_\sg\\
G&\longrightarrow&S(G).
\end{array}\right.\]
It is a Hopf algebra morphism from $(\bfH_\fg,._q,\Delta)$ to $(\bfH_\sg,._q,\Delta)$ and from $(\bfH_\fg,._0,\delta)$ to
$(\bfH_\sg,._0,\delta)$.
\end{enumerate}\end{theo}

\begin{remark}
As simple graphs are Feynman graphs, there is a canonical injection $\kappa:H_\sg\longrightarrow \fg$. 
It is compatible with $._q$ if, and only if, $q=0$. Indeed, the disjoint union product $._0$ is compatible with 
$\kappa$. If $q\neq 0$, taking $G=H=\xymatrix{\rond{}\ar[r]&\rond{}}$:
\begin{align*}
&\mbox{In $H_\fg$},&
G._q H&=\xymatrix{\rond{}\ar[r]&\rond{}}\:\xymatrix{\rond{}\ar[r]&\rond{}}+q\:\xymatrix{\rond{}&\rond{}\ar[r]\ar[l]&\rond{}}
+2q\:\xymatrix{\rond{}\ar[r]&\rond{}\ar[r]&\rond{}}\\[2mm]
&&&+q\:\xymatrix{\rond{}\ar[r]&\rond{}&\rond{}\ar[l]}+q^2\:\xymatrix{\rond{}\ar@/_1pc/[r]\ar@/^1pc/[r]&\rond{}}
+q^2\:\xymatrix{\rond{}\ar@/_1pc/[r]&\rond{}\ar@/_1pc/[l]},\\[2mm]
&\mbox{In $H_\sg$},&
G._q H&=\xymatrix{\rond{}\ar[r]&\rond{}}\:\xymatrix{\rond{}\ar[r]&\rond{}}+q\:\xymatrix{\rond{}&\rond{}\ar[r]\ar[l]&\rond{}}
+2q\:\xymatrix{\rond{}\ar[r]&\rond{}\ar[r]&\rond{}}\\[2mm]
&&&+q\:\xymatrix{\rond{}\ar[r]&\rond{}&\rond{}\ar[l]}+q^2\:\xymatrix{\rond{}\ar[r]&\rond{}}
+q^2\:\xymatrix{\rond{}\ar@/_1pc/[r]&\rond{}\ar@/_1pc/[l]}.
\end{align*}
Hence, $\kappa$ is not compatible with $._q$. 
Moreover, $\kappa$ is not compatible with $\Delta$ and $\delta$. For example, if $G=\xymatrix{\rond{}\ar[r]&\rond{}}$:
\begin{align*}
&\mbox{In $H_\fg$},&\Delta(G)&=G\otimes 1+1\otimes G+\xymatrix{\rond{}\ar[r]&}\otimes\xymatrix{\ar[r]&\rond{}},\\
&&\delta(G)&=G\otimes \xymatrix{\rond{}\ar[r]&}\xymatrix{\ar[r]&\rond{}}+\xymatrix{\rond{}}\hspace{2mm}\xymatrix{\rond{}}\otimes G,\\[2mm]
&\mbox{In $H_\sg$},&\Delta(G)&=G\otimes 1+1\otimes G+\xymatrix{\rond{}} \otimes \xymatrix{\rond{}},\\
&&\delta(G)&=G\otimes \xymatrix{\rond{}}\hspace{2mm}\xymatrix{\rond{}}+\xymatrix{\rond{}}\hspace{2mm}\xymatrix{\rond{}}\otimes G.
\end{align*}\end{remark}

\subsection{Simple graphs with no cycle}

\begin{lemma}
Let $G$ be a simple graph. The following conditions are equivalent:
\begin{enumerate}
\item For any ordered alphabet, $M_G(X)=0$.
\item $G$ has a cycle.
\end{enumerate}
Moreover, there exists an ordered alphabet $X$, such that for any simple graph with no cycle, $M_G(X)\neq 0$.
\end{lemma}

\begin{proof} Similar to the proof of Lemma \ref{graphessanscycles}. \end{proof}

Consequently, we obtain Hopf algebras structures on simple graphs with no cycle:

\begin{theo}\label{theosgwc}
Let $H_{\ncsg}$ be the vector space generated by the set of simple graphs with no cycle and let $q\in \K$. 
It is given a Hopf algebra structure: for any simple graphs with no cycle $G$ and $H$,
\begin{align}
\label{E12}G._qH&=\sum_{\substack{\sigma:V(G)\supseteq A\hookrightarrow V(H),\\ \mbox{\scriptsize $G\sqcup_\sigma H$ with no cycle}}}
q^{|A|} S(G\sqcup_\sigma H),\\
\label{E13}\Delta(G)&=\sum_{A\in I(G)}S(G_{\mid V(G)\setminus A})\otimes S(G_{\mid A}).
\end{align}
Let us consider the map:
\[T:\left\{\begin{array}{rcl}
H_\sg&\longrightarrow&H_\ncsg\\
G&\longrightarrow&\begin{cases}
G\mbox{ if $G$ has no cycle},\\
0\mbox{ otherwise}.
\end{cases}
\end{array}\right.\]
It is a Hopf algebra morphism from $(H_\sg,._q,\Delta)$ to $(H_\ncsg,._q,\Delta)$.
\end{theo}

Here is its non-commutative version:

\begin{theo}
Let $\bfH_{\ncsg}$ be the vector space generated by the set of ordered simple graphs with no cycle and let $q\in \K$. 
It is given a Hopf algebra structure by (\ref{E12})-(\ref{E13}).
Let us consider the map:
\[T:\left\{\begin{array}{rcl}
\bfH_\fg&\longrightarrow&\bfH_\ncfg\\
G&\longrightarrow&\begin{cases}
G\mbox{ if $G$ has no cycle},\\
0\mbox{ otherwise}.
\end{cases}
\end{array}\right.\]
It is a Hopf algebra morphism from $(\bfH_\sg,._q,\Delta)$ to $(\bfH_\ncsg,._q,\Delta)$ and from $(\bfH_\sg,._0,\delta)$ to
$(\bfH_\ncsg,._0,\delta)$.
\end{theo}

\subsection{quasiposets}

Let $X$ be a quasi-ordered alphabet. We consider the following quotients of $A'_q(X)$ and $\bfA'_q(X)$:
\begin{align*}
A''_q(X)&=\frac{A'_q(X)}{\langle x_{i,j}x_{j,k}(x_{i,k}-1), i\leq_X j\leq_X k\rangle},\\
\bfA''_q(X)&=\frac{\bfA'_q(X)}{\langle x_{i,j}x_{j,k}(x_{i,k}-1), i\leq_Xj\leq_X k\rangle}.
\end{align*}
The canonical surjection from $A'_q(X)$ to $A''_q(X)$ and from $\bfA'_q(X)$ to $\bfA''_q(X)$ are both denoted by $\varpi''_X$.

\begin{defi}
Let $G$ be a simple oriented graph, which set of vertices is denoted by $V(G)$. We define a quasi-order on $V(G)$, by:
\[\forall i,j\in V(G),\: i\leq_G j\mbox{ if there exists a path from $i$ to $j$ in $G$}.\] 
\end{defi}

Note that any quasi-poset, that is to say any pair $P=(V(P),\leq_P)$, where $V(P)$ is a finite set and $\leq_P$ is a quasi-order on $V(P)$,
can be obtained in this way: consider the arrow diagram $G$ of $P$, which is a simple graph such that $\leq_G=\leq_P$.

\begin{lemma}
Let $X$ be a quasi-ordered alphabet. For any simple graphs $G$, $H$, the following conditions are equivalent:
\begin{enumerate}
\item $\varpi''_X(M_G(X))=\varpi''_X(M_H(X))$ for any quasi-ordered alphabet $X$.
\item The quasi-posets $(V(G),\leq_G)$ and $(V(H),\leq_H)$ are isomorphic.
\end{enumerate}\end{lemma}

\begin{proof}
We define a congruence on the set of monomials of $A'_q(X)$ by $x_{i,j}x_{j,k}x_{i,k}\equiv x_{i,j}x_{j,k}$.
Then two monomials of $A'_q(X)$ have the same image under $\varpi''_X$ if, and only if, they are congruent.
At the level of graphs, this gives that for any quasi-ordered alphabet $X$, $\varpi''_X(M_G(X))=\varpi''_X(M_H(X))$ if, and only if, one can go
from $G$ to $H$ by a sequence of transformations:
\[(\xymatrix{i\ar[r]\ar@/^1pc/[rr]&j\ar[r]&k}) \longleftrightarrow (\xymatrix{i\ar[r]&j\ar[r]&k})\]
that is to say if, and only if, $\leq_G$ and $\leq_H$ are isomorphic. \end{proof}

\begin{prop}
\begin{enumerate}
\item Let $X,Y$ be two quasi-ordered alphabets. 
There exists a unique algebra morphism $\Delta_{X,Y}:A''_q(X\sqcup Y)\longrightarrow A''_q(X)\otimes A''_q(Y)$
such that  the following diagram commutes:
\begin{align*}
&\xymatrix{A'_q(X\sqcup Y)\ar[d]_{\varpi'_{X\sqcup Y}} \ar[r]^{\Delta_{X,Y}}&A'_q(X)\otimes A'_q(Y)\ar[d]^{\varpi'_X\otimes \varpi'_Y}\\
A''_q(X\sqcup Y)\ar[r]_{\Delta_{X,Y}}&A''_q(X)\otimes A''_q(Y)}
\end{align*}
\item The same assertion hold, after replacing $A_q$ and $A'_q$ by $\bfA_q$ and $\bfA'_q$ everywhere.
\end{enumerate}\end{prop}

\begin{proof} 1. If $i\leq_{X\sqcup Y} j\leq_{X\sqcup Y} k$, in $A'_q(X\sqcup Y)$:
\begin{align*}
\Delta_{X,Y}(x_{i,j}x_{j,k}(x_{j,k}-1))&=\begin{cases}
x_{i,j}x_{j,k}(x_{j,k}-1)\otimes 1\mbox{ if }i,k\in X,\\
1\otimes x_{i,j}x_{j,k}(x_{j,k}-1)\mbox{ if }i,k\in Y,\\
0\mbox{ if }i\in i\in X,k\in Y.
\end{cases}\end{align*}
So $\Delta_{X,Y}$ is defined from $A''_q(X\sqcup Y)$ to $A''_q(X)\otimes A''_q(Y)$.\\

2. Similar proof. \end{proof}

\begin{remark} \label{remarquecoproduit}
Unfortunately, this does not work for $\delta_{X,Y}$, except if $X$ is a totally ordered alphabet.
\end{remark}

For any quasi-poset $P$, we denote $M_P(X)=\varpi''_X(M_G(X))$ and $\bfM_G(X)=\varpi''_X(\bfM_G(X))$, 
where $G$ is any simple graph such that $\leq_G=\leq_P$, for example the arrow graph or the Hasse graph of $\leq_P$.
These elements, if all non-zero, are a basis of a subalgebra of $A''_q(X)$ or $\bfA''_q(X)$.
We obtain a quotient of $H_\sg$ and $\bfH_\sg$ based on quasi-posets. We shall need the following definitions to describe it:

\begin{defi}
Let $P,Q$ be two quasi-posets, $A\subseteq V(P)$ and $\sigma:A\hookrightarrow V(Q)$ be an injective map. We consider the quotient,
already used in Definition \ref{defi20}:
\[V(P)\sqcup_\sigma V(Q)=(V(P)\sqcup V(Q))/(a=\sigma(a),a\in A).\]
We define a relation $\mathcal{R}$ on $V(P) \sqcup V(Q)$ by:
\[\forall i,j\in V(P)\sqcup_\sigma V(Q),\: i\mathcal{R}j \mbox{ if } (i,j\in V(P), i\leq_P j)\mbox{ or }(i,j\in V(Q), i\leq_Q j).\]
The transitive closure of $\mathcal{R}$ is denoted by $\leq_{P\sqcup_\sigma Q}$, 
and $P\sqcup_\sigma Q=(V(P)\sqcup_\sigma V(Q),\leq_{P\sqcup_\sigma Q})$ is a quasi-poset.
\end{defi}

Note that if $P$ and $Q$ are ordered quasi-posets, then $P\sqcup_\sigma Q$ is also an ordered quasi-poset.

\begin{defi}
Let $P$ be a quasi-poset and let $A\subseteq V(P)$.
\begin{enumerate}
\item We denote by $P_{\mid A}$ the quasi-poset $(A,(\leq_P)_{\mid A})$. Note that if $P$ is a poset, $P_{\mid A}$ is too.
\item We shall say that $A$ is an ideal (or an open set) of $P$ if:
\begin{align*}
&\forall i,j\in V(P),&(i\in A\mbox{ and }i\leq_P j)\Longrightarrow j\in A.
\end{align*}
The set of ideals of $P$ is denoted by $I(P)$.
\end{enumerate}
\end{defi}

\begin{theo}
Let $H_{\qp}$ be the vector space generated by the set of quasi-posets and let $q\in \K$. 
\begin{enumerate}
\item It is given a Hopf algebra structure: for any quasi-posets $P,Q$,
\begin{align}
\label{E15}P._qQ&=\sum_{\sigma:V(P)\supseteq A\hookrightarrow V(Q)} q^{|A|}P\sqcup_\sigma Q,\\
\label{E16}\Delta(G)&=\sum_{A\in I(P)}P_{\mid V(P)\setminus A}\otimes P_{\mid A}.
\end{align}
\item Let us consider the map:
\[P:\left\{\begin{array}{rcl}
H_\sg&\longrightarrow&H_\qp\\
G&\longrightarrow&p(G)=(V(G),\leq_G)
\end{array}\right.\]
It is a Hopf algebra morphism from $(H_\sg,._q,\Delta)$ to $(H_\qp,._q,\Delta)$.
\end{enumerate}\end{theo}

Here is its non-commutative version:

\begin{theo}
Let $\bfH_{\qp}$ be the vector space generated by the set of ordered quasi-posets and let $q\in \K$. 
\begin{enumerate}
\item It is given a Hopf algebra structure by (\ref{E15})-(\ref{E16}).
\item Let us consider the map:
\[P:\left\{\begin{array}{rcl}
\bfH_\sg&\longrightarrow&\bfH_\qp\\
G&\longrightarrow&p(G)=(V(G),\leq_G)
\end{array}\right.\]
It is a Hopf algebra morphism from $(\bfH_\sg,._q,\Delta)$ to $(\bfH_\qp,._q,\Delta)$.
\end{enumerate}\end{theo}

The Hopf algebras $(H_\qp,._0,\Delta)$ and $\bfH_\qp,._0,\Delta)$ are introduced and studied in \cite{FM,FMP1,FMP2,Foissy2}.

\subsection{Posets}

\begin{lemma}
Let $P$ be a quasi-poset. The following conditions are equivalent:
\begin{enumerate}
\item For any ordered alphabet $X$, $M_P(X)=0$.
\item $P$ is a poset.
\end{enumerate}
Moreover, there exists an ordered alphabet $X$, such that for any poset $P$, $M_P(X)\neq 0$.
\end{lemma}

\begin{proof} Similar to the proof of Lemma \ref{graphessanscycles}. \end{proof}

Consequently, we obtain a Hopf algebra structure on posets. 

\begin{theo} \label{theoposets}
Let $H_\p$ be the vector space generated by the set of posets and let $q\in \K$. 
\begin{enumerate}
\item It is given a Hopf algebra structure: for any posets $P$ and $Q$,
\begin{align}
\label{E17}P._qQ&=\sum_{\substack{\sigma:V(P)\supseteq A\hookrightarrow V(Q),\\ \mbox{\scriptsize $P\sqcup_\sigma Q$ poset}}}
q^{|A|} S(P\sqcup_\sigma Q),\\
\label{E18}\Delta(P)&=\sum_{A\in I(P)}P_{\mid V(P)\setminus A}\otimes P_{\mid A}.
\end{align}
\item Let us consider the map:
\[P:\left\{\begin{array}{rcl}
H_\ncsg&\longrightarrow&H_\p\\
G&\longrightarrow&P(G)=(V(G),\leq_G).
\end{array}\right.\]
It is a Hopf algebra morphism from $(H_\ncsg,._q,\Delta)$ to $(H_\p,._q,\Delta)$.
\end{enumerate}\end{theo}

Here is its non-commutative version:

\begin{theo}
Let $\bfH_\p$ be the vector space generated by the set of ordered posets and let $q\in \K$. 
\begin{enumerate}
\item It is given a Hopf algebra structure by (\ref{E17})-(\ref{E18}).
\item Let us consider the map:
\[P:\left\{\begin{array}{rcl}
\bfH_\ncsg&\longrightarrow&\bfH_\p\\
G&\longrightarrow&P(G)=(V(G),\leq_G).
\end{array}\right.\]
It is a Hopf algebra morphism from $(\bfH_\ncsg,._q,\Delta)$ to $(\bfH_\p,._q,\Delta)$.
\end{enumerate}\end{theo}

\begin{remark} Using Remark \ref{remarquecoproduit}, we could use totally ordered alphabets to define 
a second coproduct on $H_\p$. We would obtain the coproduct given for any poset $P$ of order $n$ by:
\[\delta(P)=P\otimes \tun^n,\]
which is coassocative, with a right counit but no left counit.
\end{remark}

\subsection{Dual product}

We now describe the dual product of $\Delta$ on $H_\p$. 
We identify $H_\p$ and its graded dual through the symmetric pairing defined for any pair $(P,Q)$ of posets by:
\begin{align*}
\langle P,Q\rangle&=s_P \delta_{P,Q},
\end{align*}
where $s_P$ is the number of automorphisms of $P$.

\begin{prop}
Let us define a coproduct $\blacktriangle$ on $\h_\p$ in the following way: If $P$ is a poset, 
denoting by $P_1,\ldots,P_k$ its connected components,
\[\blacktriangle(P)=\sum_{I\subseteq [k]} \prod_{i\in I} P_i\otimes \prod_{i\notin I} P_i.\]
Then $\blacktriangle$ is coassociative and counitary, and for any $x,y,z\in H_\p$:
\[\langle x,y._0 z\rangle=\langle\blacktriangle(x),y\otimes z\rangle.\]
\end{prop}

\begin{proof} The coassociativity of $\blacktriangle$ is immediate. Its counit is the map defined by $\varepsilon(P)=\delta_{P,1}$ for any poset $P$.
Let $P$, $Q$, $R$ be three posets. Let us denote by $P_1,\ldots,P_k$ the different isoclasses of connected components of $P$, $Q$ and $R$.
There exist $\alpha=(\alpha_i)_{i\in [k]}$, $\beta=(\beta_i)_{i\in [k]}$, $\gamma=(\gamma_i)_{i\in [k]}$, such that:
\begin{align*}
P&=\prod_{i\in [k]} P_i^{\alpha_i},&Q&=\prod_{i\in [k]} P_i^{\beta_i},&R&=\prod_{i\in [k]} P_i^{\gamma_i}.
\end{align*} 
Moreover:
\[s_P=\prod_{i\in [k]} s_{P_i}^{\alpha_i} \alpha_i!.\]
Hence:
\[\langle P,QR\rangle=\delta_{\alpha,\beta+\gamma} \prod_{i\in [k]} s_{P_i}^{\alpha_i} \alpha_i!.\]
Moreover:
\begin{align*}
\blacktriangle(P)&=\sum_{\alpha'+\alpha''=\alpha} \prod_{i\in [k]} \frac{\alpha_i!}{\alpha'_i!\alpha''_i!} \prod_{i\in [k]} P_i^{\alpha'_i}\otimes \prod_{i\in [k]} P_i^{\alpha''_i},\\
\langle \blacktriangle(P),Q\otimes R\rangle&=
\sum_{\alpha'+\alpha''=\alpha} \prod_{i\in [k]} \frac{\alpha_i!}{\alpha'_i!\alpha''_i!}\alpha'_i! \alpha''_i! s_{P_i}^{\alpha'_i+\alpha''_i}\delta_{\alpha',\beta}\delta_{\alpha'',\gamma}\\
&=\prod_{i\in [k]} \alpha_i! s_{P_i}^{\alpha_i} \delta_{\alpha,\beta+\gamma}.
\end{align*}
Finally, $\langle \blacktriangle(P),Q\otimes R\rangle=\langle P,QR\rangle$. \end{proof}

\begin{defi}
Let $P=(V(P),\leq_P)$ and $Q=(V(Q),\leq_Q)$. We denote by $\mathcal{P}(V(Q))$ the set of subsets of $V(Q)$.
A \emph{system of edges} from $P$ to $Q$ is a map $\Theta:V(P)\longrightarrow \mathcal{P}(V(Q))$
such that:
\begin{enumerate}
\item For any $x,y\in V(P)$, such that $x<_P y$, then for any $x'\in \Theta(x)$, $y'\in \Theta(y)$, we do not have $y'\leq_Q x'$.
\item For any $x\in V(P)$, for any $x',x''\in \Theta(x)$, then $x'\leq_Q x''$ if, and only if, $x'=x''$.
\end{enumerate}\end{defi}

\begin{prop}
Let $P$, $Q$ be two posets and $\Theta$ be a system of edges from $P$ to $Q$. We define an order on $V(P)\sqcup V(Q)$ in the following way:
\begin{itemize}
\item For any $x,y\in V(P)$, $x\leq y$ if $x\leq_P y$.
\item For any $x,y\in V(Q)$, $x\leq y$ if $x\leq_Q y$.
\item For any $x\in V(P)$, $y\in V(Q)$, $x\leq y$ if there exists $x'\in V(P)$, $y'\in \Theta(x')$, such that $x\leq_P x'$ and $y'\leq_Q y$. 
\end{itemize}
This poset is denoted by $P\sqcup_\Theta Q$. Moreover, $V(Q)$ is an ideal of $P\sqcup_\Theta Q$ and the edges of the Hasse graph of $P\sqcup_\Theta Q$ are:
\begin{itemize}
\item The edges of the Hasse graph of $P$,
\item The edges of the Hasse graph of $Q$,
\item The edges $(x,x')$, with $x\in V(P)$ and $x'\in \Theta(x)$. 
\end{itemize}
\end{prop}

\begin{proof} Note that if $x\leq y$, we cannot have $x\in V(Q)$ and $y\in V(P)$.

$\leq$ is obviously reflexive. If $x\leq y$ and $y \leq x$, then both $x$ and $y$ belong to $V(P)$, or both belong to $V(Q)$.
Hence, $x\leq_P y$ and $y\leq_P x$, or  $x\leq_Q y$ and $y\leq_Q x$. In both cases, $x=y$, so $\leq$ is antisymmetric.
Let us assume that $x\leq y$ and $y\leq z$. Four cases hold.
\begin{enumerate}
\item $(x,y,z)\in V(P)^3$. Then $x\leq_P y\leq_P z$, so $x\leq z$.
\item $(x,y,z)\in V(P)^2\times V(Q)$. There exist $y'\in V(P)$, $z'\in \Theta(x')$, such that $x\leq_P y\leq_P y'$ and $z'\leq_Q z$. So $x \leq z$.
\item $(x,y,z)\in V(P)\times V(Q)^2$. There exist $x'\in V(P)$, $y'\in \Theta(x')$, such that $x\leq_P x'$ and $y'\leq_Q y\leq_Q z$. So $x\leq z$.
\item $(x,y,z)\in V(Q)^3$. Then $x\leq_Q y\leq_Q z$, so $x\leq z$.
\end{enumerate}
Hence, $\leq$ is an order. \\

Let $(x,y)$ be an edge of the Hasse graph of $P\sqcup_\Theta Q$, that is to say:
\begin{itemize}
\item $x<y$.
\item If $x\leq z\leq y$, then $z\in \{x,y\}$.
\end{itemize}
Three cases are possible.
\begin{enumerate}
\item $(x,y)\in V(P)^2$. As $(P\sqcup Q)_{\mid V(P)}=P$, $(x,y)$ is an edge of the Hasse graph of $P$.
\item $(x,y)\in V(Q)^2$. As $(P\sqcup Q)_{\mid V(Q)}=Q$, $(x,y)$ is an edge of the Hasse graph of $Q$.
\item $(x,y)\in V(P)\times V(Q)$. Then there exists $x'\in V(P)$, $y'\in \Theta(x')$, such that $x\leq_P x'$ and $y'\leq_Q y$.
By definition, $x'<y',$ so $x\leq x'<y'\leq y$. As $(x,y)$ is an edge, $x=x'$ and $y=y'$, so $y\in \Theta(x)$.
\end{enumerate}
Conversely:
\begin{itemize}
\item Let $(x,y)$ be an edge of the Hasse graph of $P$. If $x\leq z\leq y$, necessarily $z\in V(P)$ as $y\in V(P)$, so $x\leq_P z\leq_P y$,
which implies $z\in \{x,y\}$.
\item Similarly, if $(x,y)$ is an edge of the Hasse graph of $Q$, then it is an edge of the Hasse graph of $P\sqcup_\Theta Q$.
\item If $x\in V(P)$ and $y\in \Theta(x)$, then $x< y$. If $x\leq z\leq y$, two cases are possible.
\begin{enumerate}
\item If $z\in V(P)$, there exists $z'\in V(P)$, $y'\in \Theta(z')$, such that $x\leq_P z\leq_P z'$, $y'\leq_Q y$. So $x\leq z\leq z'<y'\leq y$.
As $(x,y)$ is an edge, $x=z=z'$ and $y'=y$.
\item Similarly, if $z\in V(Q)$, $z=y$.
\end{enumerate}
\end{itemize}
Obviously, $V(Q)$ is an ideal of $P\sqcup_\Theta Q$. \end{proof}

\begin{lemma}
Let $P_1,P_2,P$ be three posets. We consider the two following sets:
\begin{itemize}
\item $C(P_1,P_2,P)$ is the set of triple $(I,\phi_1,\phi_2)$, where $I$ is an ideal of $P$, $\phi_1$ is an isomorphism from $P_1$ to $P_{\mid V(P)\setminus I}$
and $\phi_2$ is an isomorphism from $P_2$ to $P_{\mid I}$.
\item $D(P_1,P_2,P)$ is the set of pairs $(\Theta,\phi)$, where $\Theta$ is a system of edges from $P_1$ to $P_2$ and $\phi$
is an isomorphism from $P_1\sqcup_\Theta P_2$ to $P$.
\end{itemize}
Then $C(P_1,P_2,P)$ and $D(P_1,P_2,P)$ are in bijection.
\end{lemma}

\begin{proof} We shortly denote $C=C(P_1,P_2,P)$ and $D=D(P_1,P_2,P)$.\\

\textit{First step.} Let $F:C\longrightarrow D$ defined by $F(I,\phi_1,\phi_2)=(\Theta,\phi)$, with:
\begin{itemize}
\item $\phi$ is defined by $\phi_{\mid V(P_1)}=\phi_1$ and $\phi_{\mid V(P_2)}=\phi_2$.
\item For any $x\in V(P_1)$, $\Theta(x)$ is the set of $y\in V(P_2)$ such that $(\phi_1(x),\phi_2(y))$ is an edge of the Hasse graph of $P$. 
\end{itemize}
Let us prove that $F$ is well-defined.

Let $x,y\in V(P_1)$, $x'\in \Theta(x)$, $y'\in \Theta(y)$, such that $x<_{P_1} y$ and $y'\leq_{P_2}x'$. Then $\phi_1(x)<_P \phi_1(y)$,
$\phi_2(y')\leq_P \phi_2(x')$, $(\phi_1(x),\phi_2(x'))$ and $(\phi_1(y),\phi_1(y'))$ are edges of the Hasse graph of $P$. We obtain:
\[\phi_1(x)<_P \phi_1(y)<_P\phi_2(y')\leq_P \phi_2(x').\]
This contradicts that $(\phi_1(x),\phi_2(x'))$ is an edge of the Hasse graph of $P$. 

Let $x\in V(P_1)$, $x',x''\in \Theta(x)$, with $x'\leq_{P_2} x''$. Then $\phi_1(x)<_P \phi_2(x')\leq \phi_2(x'')$.
As $(\phi_1(x),\phi_2(x'')$ is an edge, $\phi_2(x')=\phi_2(x'')$, so $x'=x''$. We proved that $\Theta$ is a system of edges from $P_1$ to $P_2$.\\

By the preceding lemma, the image by $\phi$ of the edges of $P_1\sqcup_\Theta P_2$ are the edges of $P$, so $\phi$ is an isomorphism.
We proved that $F$ is well-defined.\\

\textit{Second step.} Let $G:D\longrightarrow C$ defined by $G(\Theta,\phi)=(I,\phi_1,\phi_2)$, where:
\begin{itemize}
\item $I=\phi(V(P_2))$.
\item $\phi_1=\phi_{\mid V(P_1)}$ and $\phi_2=\phi_{\mid V(P_2)}$.
\end{itemize}
By the preceding lemma, $G$ is well-defined. \\

\textit{Last step.} Let $(I,\phi_1,\phi_2)\in C$. We put $F(I,\phi_1,\phi_2)=(\Theta,\phi)$ and $G(\Theta,\phi)=(I',\phi_1',\phi_2')$.
Then $I'=\phi(V(P_2))=I$, $\phi'_1=\phi_{\mid V(P_1)}=\phi_1$ and similarly, $\phi'_2=\phi_2$. So $G\circ F=Id_C$.

Let $(\Theta,\phi)\in D$. We put $G(\Theta,\phi)=(I,\phi_1,\phi_2)$ and $F(I,\phi_1,\phi_2)=(\Theta',\phi')$. For any $x\in V(P_1)$, $y\in V(P_2)$:
\begin{align*}
y\in \Theta'(x)&\Longleftrightarrow (\phi_1(x),\phi_2(y))\mbox{ edge of the Hasse graph of $P$}\\
&\Longleftrightarrow (\phi(x),\phi(y))\mbox{ edge of the Hasse graph of $P$}\\
&\Longleftarrow (x,y)\mbox{ edge of the Hasse graph of $P\sqcup_\Theta Q$}\\
&\Longleftrightarrow y\in \Theta(x).
\end{align*}
Therefore, $\Theta'=\Theta$. Moreover, for $i\in \{1,2\}$, $\phi'_{\mid V(P_i)}=\phi_i=\phi_{\mid V(P_i)}$, so $\phi'=\phi$. Hence, $F\circ G=Id_D$.
\end{proof}

\begin{theo}
We define a product $\star$ on $H_\p$ in the following way: for any posets $P,Q$,
\[P\star Q=\sum_{\substack{\mbox{\scriptsize $\Theta$ system of edges}\\ \mbox{\scriptsize from $P$ to $Q$}}} P \sqcup_\Theta Q.\]
Then $(\h_\p,\star,\blacktriangle)$ is a Hopf algebra. Moreover, for any $x,y,z\in H_\p$:
\[\langle x\star y,z\rangle=\langle x\otimes y,\Delta(z)\rangle.\]
\end{theo}

\begin{proof} Let $P_1,P_2,P$ be posets. By the preceding lemma:
\begin{align*}
\langle P_1\otimes P_2,\Delta(P)\rangle&=\sum_{\mbox{\scriptsize $I$ ideal of $P$}} \langle P_1,P_{\mid V(P)\setminus I}\rangle\langle P_2,P_{\mid I}\rangle\\
&=\sharp C(P_1,P_2,P)\\
&=\sharp D(P_1,P_2,P)\\
&=\sum_{\substack{\mbox{\scriptsize $\Theta$ system of edges}\\ \mbox{\scriptsize from $P_1$ to $P_2$}}} \langle P_1\sqcup_\Theta P_2,P\rangle\\
&=\langle P_1\star P_2,P\rangle.
\end{align*}
As $\langle-,-\rangle$ is non degenerate and $(H_\p,m,\Delta)$ is a Hopf algebra, dually $(\h_\p,\star,\blacktriangle)$ is a Hopf algebra.

Let $x,y,z\in H_\p$. For any $t\in H_\p$:
\begin{align*}
\langle (x\star y)*z,t\rangle&=\langle x\otimes y\otimes z,(\Delta \otimes Id)\circ \delta(t)\rangle\\
&=\langle x\otimes y\otimes z,m_{1,3,24}\circ (\delta \otimes \delta)\circ \Delta(t)\rangle\\
&=\langle x\otimes z'\otimes y\otimes z'',(\delta\otimes \delta)\circ \Delta(t)\rangle\\
&=\langle (x*z')\star (x*z''),t\rangle.
\end{align*}
We conclude by the non degeneracy of $\langle-,-\rangle$. \end{proof}

\bibliographystyle{amsplain}
\bibliography{biblio}

\end{document}